\documentclass[x11names,11pt,twoside]{amsart}

\usepackage[utf8]{inputenc}
\usepackage[english]{babel}
\usepackage[T1]{fontenc}
\usepackage{amsfonts}
\usepackage{geometry}
\usepackage{tikz}
\tikzstyle{every picture}=[line width=.7pt,minimum size=3pt,every label/.append style={font=\normalsize},label distance=2pt]
\tikzstyle{every node}=[font=\normalsize,circle,draw=black,fill=black,inner sep=0pt,minimum width=1.3pt]
\usepackage{amsmath, amsthm, amssymb}
\usepackage{latexsym}
\usepackage{enumitem}
\usepackage{fancyhdr}
\usepackage{mathtools}
\usepackage{caption}
\usepackage[labelformat=simple]{subcaption}

\usepackage{color}
\usepackage{bm}
\usepackage[bookmarks=true]{hyperref}

\makeatletter
\newtheorem*{rep@theorem}{\rep@title}
\newcommand{\newreptheorem}[2]{%
\newenvironment{rep#1}[1]{%
 \def\rep@title{#2 \ref{##1}}%
 \begin{rep@theorem}}%
 {\end{rep@theorem}}}
\makeatother

\theoremstyle{plain}
\newtheorem{theorem}{Theorem}[section]
\newtheorem{proposition}[theorem]{Proposition}
\newtheorem{corollary}[theorem]{Corollary}
\newtheorem{lemma}[theorem]{Lemma}
\theoremstyle{definition}
\newtheorem{definition}[theorem]{Definition}
\newtheorem{example}[theorem]{Example}

\newtheorem{remark}[theorem]{Remark}
\newreptheorem{theorem}{Theorem}

\date{}

\newcommand{\hgt}{{\rm ht}}

\newcommand{\depth}{{\rm depth}}

\newcommand{\mc}{\mathcal}

\title{Binomial edge ideals of bipartite graphs}

\author{Davide Bolognini, Antonio Macchia, Francesco Strazzanti}

\address{{\small Davide Bolognini, Institut f\"{u}r Mathematik, Albrechtstrasse 28a, 49076 Osnabr\"{u}ck, Germany}}
\email{{\small davide.bolognini@yahoo.it}}

\address{{\small Antonio Macchia, Dipartimento di Matematica, Università degli Studi di Bari ``Aldo Moro'', Via Orabona 4, 70125 Bari, Italy}}
\email{{\small macchia.antonello@gmail.com}}

\address{{\small Francesco Strazzanti, Departamento de \'Algebra, Facultad de Matem\'aticas, Universidad de Sevilla, Avda. Reina Mercedes s/n, 41080 Sevilla, Spain}}
\email{{\small francesco.strazzanti@gmail.com}}

\begin{document}
\thanks{The first and the second author were supported by INdAM}
\thanks{The third author was partially supported by MTM2013-46231-P (Ministerio de Economı\'ia y Competitividad) and FEDER}

\begin{abstract}
We classify the bipartite graphs $G$ whose binomial edge ideal $J_G$ is Cohen-Macaulay. The connected components of such graphs can be obtained by gluing a finite number of basic blocks with two operations. In this context we prove the converse of a well-known result due to Hartshorne, showing that the Cohen-Macaulayness of these ideals is equivalent to the connectedness of their dual graphs. We study interesting properties also for non-bipartite graphs and in the unmixed case, constructing classes of bipartite graphs with $J_G$ unmixed and not Cohen-Macaulay.
\end{abstract}


\maketitle

\section{Introduction}\label{S.Introduction}
Binomial edge ideals were introduced independently in \cite{HHHKR10} and \cite{O11}. They are a natural generalization of the ideals of $2$-minors of a $(2 \times n)$-generic matrix \cite{BV88}: their generators are those $2$-minors whose column indices correspond to the edges of a graph. In this perspective, the ideals of $2$-minors are binomial edge ideals of complete graphs. On the other hand, binomial edge ideals arise naturally in Algebraic Statistics, in the context of conditional independence ideals, see \cite[Section 4]{HHHKR10}.

More precisely, given a finite simple graph $G$ on the vertex set $[n]=\{1,\dots,n\}$, the \textit{binomial edge ideal} associated with $G$ is the ideal
\[
J_G=(x_iy_j-x_jy_i : \{i,j\} \text{ is an edge of } G) \subset R=K[x_i,y_i : i \in [n]].
\]

Binomial edge ideals have been extensively studied, see e.g. \cite{BN17}, \cite{EHH11}, \cite{EHH14}, \cite{KM15}, \cite{KM16}, \cite{MM13}, \cite{RR14}, \cite{R13}. Yet a number of interesting questions is still unanswered. In particular, many authors have studied classes of Cohen-Macaulay binomial edge ideals in terms of the associated graph, see e.g. \cite{BN17}, \cite{EHH11}, \cite{KM15}, \cite{RR14}, \cite{R13}. Some of these results concern a class of chordal graphs, the so-called \textit{closed graphs}, introduced in \cite{HHHKR10}, and their generalizations, such as block and generalized block graphs \cite{KM15}.

In the context of squarefree monomial ideals, any graph can be associated with the so-called \textit{edge ideal}, whose generators are monomials of degree $2$ corresponding to the edges of the graph. Herzog and Hibi, in \cite[Theorem 3.4]{HH05}, classified Cohen-Macaulay edge ideals of bipartite graphs in purely combinatorial terms.
In the same spirit, we provide a combinatorial classification of Cohen-Macaulay binomial edge ideals of bipartite graphs. In particular, we present a family of bipartite graphs $F_m$ whose binomial edge ideal is Cohen-Macaulay, and we prove that, if $G$ is connected and bipartite, then $J_G$ is Cohen-Macaulay if and only if $G$ can be obtained recursively by gluing a finite number of graphs of the form $F_m$ via two operations.

We now explain in more detail the basic blocks and the operations in our classification.
For the terminology about graphs we refer to \cite{D10}.

\textbf{Basic blocks:} For every $m \geq 1$, let $F_m$ be the graph (see Figure \ref{F.exampleF_m}) on the vertex set $[2m]$ and with edge set
\[
E(F_m) = \left\{ \{2i,2j-1\} : i=1,\dots,m, j=i,\dots,m \right\}.
\]
Notice that $F_1$ is the single edge $\{1,2\}$ and $F_2$ is the path of length $3$.

\begin{figure}[ht!]
\begin{subfigure}[c]{0.45\textwidth}
\centering
\begin{tikzpicture}
\node[label={above:$1$}] (a) at (0,1.5) {};
\node[label={above:$3$}] (b) at (1,1.5) {};
\node[label={above:$5$}] (c) at (2,1.5) {};
\node[label={below:$2$}] (e) at (0,0) {};
\node[label={below:$4$}] (f) at (1,0) {};
\node[label={below:$6$}] (g) at (2,0) {};
\draw (0,1.5) -- (0,0) -- (1,1.5) -- (1,0) -- (2,1.5) -- (2,0)
(2,1.5) -- (0,0);
\end{tikzpicture}
\caption{The graph $F_3$}\label{F.F_3}
\end{subfigure}
\begin{subfigure}[c]{0.45\textwidth}
\centering
\begin{tikzpicture}
\node[label={above:$1$}] (a) at (0,1.5) {};
\node[label={above:$3$}] (b) at (1,1.5) {};
\node[label={above:$5$}] (c) at (2,1.5) {};
\node[label={above:$7$}] (d) at (3,1.5) {};
\node[label={below:$2$}] (e) at (0,0) {};
\node[label={below:$4$}] (f) at (1,0) {};
\node[label={below:$6$}] (g) at (2,0) {};
\node[label={below:$8$}] (h) at (3,0) {};
\draw (0,1.5) -- (0,0) -- (1,1.5) -- (1,0) -- (2,1.5) -- (2,0) -- (3,1.5) -- (3,0)
(2,1.5) -- (0,0) -- (3,1.5) -- (1,0);
\end{tikzpicture}
\caption{The graph $F_4$}\label{F.F_4}
\end{subfigure}
\caption{} \label{F.exampleF_m}
\end{figure}

\textbf{Operation $\ast$:} For $i=1,2$, let $G_i$ be a graph with at least one vertex $f_i$ of degree one, i.e., a {\em leaf} of $G_i$. We denote the graph $G$ obtained by identifying $f_1$ and $f_2$ by $G=(G_1,f_1) \ast (G_2,f_2)$, see Figure \ref{F.example2a}. This is a particular case of an operation studied by Rauf and Rinaldo in \cite[Section 2]{RR14}.

\textbf{Operation $\circ$:} For $i=1,2$, let $G_i$ be a graph with at least one leaf $f_i$, $v_i$ its neighbour and assume $\deg_{G_i}(v_i) \geq 3$. We define $G = (G_1,f_1) \circ (G_2,f_2)$ to be the graph obtained from $G_1$ and $G_2$ by removing the leaves $f_1,f_2$ and identifying $v_1$ and $v_2$, see Figure \ref{F.example2b}.

For both operations, if it is not important to specify the vertices $f_i$ or it is clear from the context, we simply write $G_1 \ast G_2$ or $G_1 \circ G_2$.

\begin{figure}[ht!]
\begin{subfigure}[c]{0.45\textwidth}
\centering
\begin{tikzpicture}
\node (a) at (1,1.5) {};
\node (b) at (2,1.5) {};
\node (c) at (3,1.5) {};
\node (d) at (0,0) {};
\node (e) at (1,0) {};
\node (f) at (2,0) {};
\node (g) at (3,0) {};
\node[label={[label distance=-3mm]270:$f_1 \!=\! f_2$}] (h) at (-1,0) {};
\node (i) at (-2,0) {};
\node (i) at (-3,0) {};
\node (i) at (-4,0) {};
\node (j) at (-3,1.5) {};
\node (k) at (-4,1.5) {};
\draw (-4,0) -- (-4,1.5) -- (-3,0) -- (-3,1.5) -- (-2,0) -- (-1,0) -- (0,0) -- (1,1.5) -- (1,0) -- (2,1.5) -- (2,0) -- (3,1.5) -- (3,0)
(2,1.5) -- (0,0) -- (3,1.5) -- (1,0)
(-4,1.5) -- (-2,0);
\end{tikzpicture}
\caption{The graph $F_3 \ast F_4$}\label{F.example2a}
\end{subfigure}
\begin{subfigure}[c]{0.45\textwidth}
\centering
\begin{tikzpicture}
\node (a) at (1,1.5) {};
\node (b) at (2,1.5) {};
\node (c) at (3,1.5) {};
\node[label={[label distance=-3mm]270:$v_1 \!=\! v_2$}] (d) at (0,0) {};
\node (e) at (1,0) {};
\node (f) at (2,0) {};
\node (g) at (3,0) {};
\node (h) at (-1,0) {};
\node (i) at (-2,0) {};
\node (j) at (-1,1.5) {};
\node (k) at (-2,1.5) {};
\draw (-2,0) -- (-2,1.5) -- (-1,0) -- (-1,1.5) -- (0,0) -- (1,1.5) -- (1,0) -- (2,1.5) -- (2,0) -- (3,1.5) -- (3,0)
(2,1.5) -- (0,0) -- (3,1.5) -- (1,0)
(-2,1.5) -- (0,0);
\end{tikzpicture}
\caption{The graph $F_3 \circ F_4$}\label{F.example2b}
\end{subfigure}
\caption{} \label{F.exampleOperations}
\end{figure}

Finally, we recall the notion of dual graph of an ideal, which is one of the main tools in the proof of our classification. We follow the notation used in \cite{BV15}.

\indent \textbf{Dual graph:}  Let $I$ be an ideal in a polynomial ring $A=K[x_1,\dots,x_n]$ and let $\mathfrak p_1,\dots,\mathfrak p_r$ be the minimal prime ideals of $I$. The \textit{dual graph} $\mathcal D(I)$ is a graph with vertex set $[r]$ and edge set
\[
\{\{i,j\} : \hgt(\mathfrak p_i+\mathfrak p_j)-1=\hgt(\mathfrak p_i)=\hgt(\mathfrak p_j)=\hgt(I)\}.
\]

This notion was originally studied by Hartshorne in \cite{H62} in terms of \textit{connectedness in codimension one}.
By \cite[Corollary 2.4]{H62}, if $A/I$ is Cohen-Macaulay, then the algebraic variety defined by $I$ is connected in codimension one, hence $I$ is unmixed by \cite[Remark 2.4.1]{H62}. The connectedness of the dual graph translates in combinatorial terms the notion of connectedness in codimension one, see \cite[Proposition 1.1]{H62}. Thus, if $A/I$ is Cohen-Macaulay, then $\mathcal D(I)$ is connected. The converse does not hold in general, see for instance Remark \ref{connessononCM}. We will show that for binomial edge ideals of connected bipartite graphs this is indeed an equivalence. In geometric terms, this means that the algebraic variety defined by $J_G$ is Cohen-Macaulay if and only if it is connected in codimension one.\\[2mm]
\indent Given a graph $G$, the ideal $J_G$ is Cohen-Macaulay if and only if the binomial edge ideal of each connected component of $G$ is Cohen–Macaulay. Thus, we may assume $G$ connected with at least two vertices.

Before stating the main result, we recall the notion of cut set, which is central in the study of binomial edge ideals. In fact, there is a bijection between the cut sets of a graph $G$ and the minimal prime ideals of $J_G$, see \cite[Section 3]{HHHKR10}. For a subset $S \subseteq [n]$, let $c_G(S)$ be the number of connected components of the induced subgraph $G_{[n] \setminus S}$. The set $S$ is called \textit{cut set} of $G$ if $S=\emptyset$ or $S \neq \emptyset$ and $c_G(S \setminus \{i\}) < c_G(S)$ for every $i \in S$. Moreover, we call \textit{cut vertex} a cut set of cardinality one. We denote by $\mathcal M(G)$ the set of cut sets of $G$.

We are now ready to state our main result.

\begin{reptheorem}{T.CHARACTERIZATION}
Let $G$ be a connected bipartite graph. The following properties are equivalent:
\begin{itemize}
  \item[{\rm a)}] $J_G$ is Cohen-Macaulay;
  \item[{\rm b)}] the dual graph $\mathcal D(J_G)$ is connected;
  \item[{\rm c)}] $G=A_1 \ast A_2 \ast \cdots \ast A_k$, where $A_i=F_m$ or $A_i=F_{m_1} \circ \cdots \circ F_{m_r}$, for some $m \geq 1$ and $m_j \geq 3$;
  \item[{\rm d)}] $J_G$ is unmixed and for every non-empty $S \in \mc M(G)$, there exists $s \in S$ such that $S \setminus \{s\} \in \mc M(G)$.
\end{itemize}
\end{reptheorem}

The paper is structured as follows. In Section \ref{S.unmixed} we study unmixed binomial edge ideals of bipartite graphs. A combinatorial characterization of unmixedness was already proved in \cite{HHHKR10} (see also \cite[Lemma 2.5]{RR14}), in terms of the cut sets of the underlying graph.

A first distinguishing fact about bipartite graphs with $J_G$ unmixed is that they have exactly two leaves (Proposition \ref{P.unmixed2leaves}). This, in particular, means that $G$ has at least two cut vertices. In Proposition \ref{Semiconi}, we present a construction that is useful in the study of the basic blocks and to produce new examples of unmixed binomial edge ideals, which are not Cohen-Macaulay.

In Section \ref{S.basicBlocks} we prove that the ideals $J_{F_m}$, associated with the basic blocks of our construction, are Cohen-Macaulay, see Proposition \ref{P.type1CM}. In Section \ref{S.gluing} we study the operations $\ast$ and $\circ$. In \cite[Theorem 2.7]{RR14}, Rauf and Rinaldo proved that $J_{G_1 \ast G_2}$ is Cohen-Macaulay if and only if so are $J_{G_1}$ and $J_{G_2}$. In Theorem \ref{T.type2bCMbipartite}, we show that $J_G$ is Cohen-Macaulay if $G=F_{m_1} \circ \cdots \circ F_{m_k}$, for every $k \geq 2$ and $m_i \geq 3$.
Using these results, we prove the implication c) $\Rightarrow$ a) of Theorem \ref{T.CHARACTERIZATION}.

Section \ref{S.dualGraph} is devoted to the study of the dual graph of binomial edge ideals. This is one of the main tools in the proof of Theorem \ref{T.CHARACTERIZATION}. First of all, given a (not necessarily bipartite) graph $G$ with $J_G$ unmixed, in Theorem \ref{Dual graph} we provide an explicit description of the edges of the dual graph $\mc D(J_G)$ in terms of the cut sets of $G$. This allows us to show infinite families of bipartite graphs whose binomial edge ideal is unmixed and not Cohen-Macaulay, see Examples \ref{E.unmixedNotCM1} and \ref{E.unmixedNotCM2}.

A crucial result concerns a basic, yet elusive, property of cut sets of unmixed binomial edge ideals. In Lemma \ref{INTERSECTION}, we show that, mostly for bipartite graphs and under some assumption, the intersection of any two cut sets is a cut set. This leads to the proof of the equivalence b) $\Leftrightarrow$ d) in Theorem \ref{T.CHARACTERIZATION}, see Theorem \ref{T.cutSetMinusVertex}.
On the other hand, if $G=G_1 \ast G_2$ or $G=G_1 \circ G_2$ is bipartite and $\mc D(J_G)$ is connected, then the dual graphs of $G_1$ and $G_2$ are connected, see Theorem \ref{T.moreThan2CutVertices}. Thus, we may reduce to consider bipartite graphs with exactly two cut vertices and prove the implication b) $\Rightarrow$ c) of Theorem \ref{T.CHARACTERIZATION}.

It is worth noting that, the main theorem gives also a classification of other classes of Cohen-Macaulay binomial ideals associated with bipartite graphs, Corollary \ref{C.otherBinomialIdeals}: \textit{Lov\'asz-Saks-Schrijver ideals} \cite{HMMW15}, \textit{permanental edge ideals} \cite[Section 3]{HMMW15} and \textit{parity binomial edge ideals} \cite{KSW16}.

As an application of the main result, in Corollary \ref{C.Hirsch}, we show that Cohen-Macaulay binomial edge ideals of bipartite graphs are Hirsch, meaning that the diameter of the dual graph of $J_G$ is bounded above by the height of $J_G$, verifying \cite[Conjecture 1.6]{BV15}.

All the results presented in this paper are independent of the field.

\section{Unmixed binomial edge ideals of bipartite graphs} \label{S.unmixed}
In this paper all graphs are \textit{finite} and \textit{simple} (without loops and multiple edges). In what follows, unless otherwise stated, we assume that $G$ is a connected graph with at least two vertices. Given a graph $G$, we denote by $V(G)$ its vertex set and by $E(G)$ its edge set. If $G$ is a {\em bipartite graph}, we denote by $V(G)=V_1 \sqcup V_2$ the {\em bipartition} of the vertex set and call $V_1,V_2$ the \textit{bipartition sets} of $G$.

For a subset $S \subseteq V(G)$, we denote by $G_S$ the \textit{subgraph induced} in $G$ by $S$, which is the graph with vertex set $S$ and edge set consisting of all the edges of $G$ with both endpoints in $S$.

We recall some definitions and results from \cite{HHHKR10}. Let $G$ be a graph with vertex set $[n]$. We denote by $R=K[x_i, y_i : i \in [n]]$ the polynomial ring in which the ideal $J_G$ is defined and, if $S \subseteq [n]$, we set $\overline S=[n] \setminus S$. Let $c_G(S)$, or simply $c(S)$, be the number of connected components of the induced subgraph $G_{\overline S}$ and let $G_1,\dots,G_{c_G(S)}$ be the connected components of $G_{\overline S}$. For each $G_i$, denote by $\widetilde G_i$ the complete graph on $V(G_i)$ and define the ideal
\vspace*{-2mm}
\[
P_S(G) = \left( \, \bigcup_{i \in S} \ \{x_i,y_i\}, J_{\widetilde G_1},\dots,J_{\widetilde G_{c_G(S)}} \right).
\]

In \cite[Section 3]{HHHKR10}, it is shown that $P_S(G)$ is a prime ideal for every $S \subseteq [n]$,  $\hgt(P_S(G))=n+|S|-c_G(S)$ and $J_G= \bigcap_{S \subseteq [n]} P_S(G)$. Moreover, $P_S(G)$ is a minimal prime ideal of $J_G$ if and only if $S=\emptyset$ or $S \neq \emptyset$ and $c_G(S \setminus \{i\}) < c_G(S)$ for every $i \in S$. In simple terms the last condition means that, adding a vertex of $S$ to $G_{\overline S}$, we connect at least two connected components of $G_{\overline S}$. We set
\begin{align*}
\mathcal M(G) &= \{S \subset [n] : P_S(G) \text{ is a minimal prime ideal of } J_G\}\\
&= \{\emptyset\} \cup \{S \subset [n] : S \neq \emptyset, \, c_G(S \setminus \{i\}) < c_G(S) \text{ for every } i \in S\},
\end{align*}
and we call \textit{cut sets} of $G$ the elements of $\mathcal M(G)$. If $\{v\} \in \mathcal M(G)$, we say that $v$ is a \textit{cut vertex} of $G$.

We further recall that a \textit{clique} of a graph $G$ is a subset $C \subseteq V(G)$ such that $G_C$ is complete. A \textit{free vertex} of $G$ is a vertex that belongs to exactly one maximal clique of $G$. A vertex of degree $1$ in $G$, which in particular is a free vertex, is called a \textit{leaf} of $G$.

\begin{remark} \label{R.leafNotCutSet}
If $v$ is a free vertex of a graph $G$, then $v \notin S$, for every $S \in \mathcal M(G)$. In fact, if $v \in S$, for some $S \in \mathcal M(G)$, then $c_G(S)=c_G(S \setminus \{v\})$.
\end{remark}

Recall that an ideal is \textit{unmixed} if all its minimal primes have the same height. By \cite[Lemma 2.5]{RR14}, $J_G$ is unmixed if and only if for every $S \in \mc{M}(G)$,
\begin{equation}\label{Eq.unmixedness}
c_G(S)=|S|+1.
\end{equation}
This follows from the equality $\hgt(P_\emptyset(G))=n-1=\hgt(P_S(G))=n+|S|-c_G(S)$.

Moreover, for every graph $G$, with $J_G$ unmixed, we have that $\dim(R/J_G)=|V(G)|+c$, where $c$ is the number of connected components of $G$, see \cite[Corollary 3.3]{HHHKR10}.

In this section, we study some properties of unmixed binomial edge ideals of bipartite graphs. It is well-known that if $J_G$ is Cohen-Macaulay, then $J_G$ is unmixed. The converse is, in general, not true, also for binomial edge ideals of bipartite graphs. In fact, in the following example we show two classes of bipartite graphs whose binomial edge ideals are unmixed but not Cohen-Macaulay.

\begin{example}\label{E.unmixedNotCM1}
For every $k \geq 4$, let $M_{k,k}$ be the graph with vertex set $[2k]$ and edge set
\[
E(M_{k,k})=\{\{1,2\},\{2k-1,2k\}\} \cup \{\{2i,2j-1\} : i=1,\dots,k-1, j=2,\dots,k\},
\]
see Figure \ref{F.M_(4,4)}, and let $M_{k-1,k}$ be the graph with vertex set $[2k-1]$ and edge set
\[
E(M_{k-1,k})=\{\{1,2\},\{2k-2,2k-1\}\} \cup \{\{2i,2j-1\} : i=1,\dots,k-1, j=2,\dots,k-1\},
\]
see Figure \ref{F.M_(3,4)}.

\begin{figure}[ht!]
\begin{subfigure}[c]{0.45\textwidth}
\centering
\begin{tikzpicture}
\node[label={above:$1$}] (a) at (0,1.5) {};
\node[label={above:$3$}] (b) at (1,1.5) {};
\node[label={above:$5$}] (c) at (2,1.5) {};
\node[label={above:$7$}] (d) at (3,1.5) {};
\node[label={below:$2$}] (e) at (0,0) {};
\node[label={below:$4$}] (f) at (1,0) {};
\node[label={below:$6$}] (g) at (2,0) {};
\node[label={below:$8$}] (h) at (3,0) {};
\draw (0,1.5) -- (0,0) -- (1,1.5) -- (1,0) -- (2,1.5) -- (2,0) -- (3,1.5) -- (3,0)
(2,1.5) -- (0,0) -- (3,1.5) -- (1,0)
(1,1.5) -- (2,0);
\end{tikzpicture}
\caption{The graph $M_{4,4}$}\label{F.M_(4,4)}
\end{subfigure}
\begin{subfigure}[c]{0.45\textwidth}
\centering
\begin{tikzpicture}
\node[label={above:$1$}] (a) at (0,1.5) {};
\node[label={above:$3$}] (b) at (1,1.5) {};
\node[label={above:$5$}] (c) at (2,1.5) {};
\node[label={above:$7$}] (d) at (3,1.5) {};
\node[label={below:$2$}] (e) at (0.5,0) {};
\node[label={below:$4$}] (f) at (1.5,0) {};
\node[label={below:$6$}] (g) at (2.5,0) {};
\draw (0,1.5) -- (0.5,0) -- (1,1.5) -- (1.5,0) -- (2,1.5) -- (2.5,0) -- (3,1.5)
(0.5,0) -- (2,1.5)
(1,1.5) -- (2.5,0);
\end{tikzpicture}
\caption{The graph $M_{3,4}$}\label{F.M_(3,4)}
\end{subfigure}
\caption{}
\end{figure}

Notice that the graphs $M_{k,k}$ and $M_{k-1,k}$ are obtained by adding two \textit{whiskers} to some complete bipartite graph. Recall that adding a whisker to a graph $G$ means adding a new vertex and connect it to one of the vertices of $G$.

Let $V_1 \sqcup V_2$ be the bipartition of $M_{k,k}$ and of $M_{k-1,k}$ such that $V_1$ contains the odd labelled vertices and $V_2$ contains the even labelled vertices. We claim that
\begin{gather*}
\mathcal M(M_{k,k}) = \{\emptyset, \{2\}, \{2k-1\}, \{2,2k-1\}, V_1 \setminus \{1\}, V_2 \setminus \{2k\}\} \text{ and} \\
\mathcal M(M_{k-1,k}) = \{\emptyset, \{2\}, \{2k-2\}, \{2,2k-2\}, V_1 \setminus \{1,2k-1\}, V_2\}.
\end{gather*}

The inclusion $\supseteq$ is clear. We prove the other inclusion for $M_{k,k}$, the proof is similar for $M_{k-1,k}$. Let $S \in \mc M(M_{k,k})$. If $S \subseteq \{2,2k-1\}$, there is nothing to prove. If there exists $v \in S \setminus \{2,2k-1\}$, then $S=V_1 \setminus \{1\}$ or $S=V_2 \setminus \{2k\}$. In fact, if $v \in V_1 \setminus \{1\}$ and there is $w \in (V_1 \setminus \{1\}) \setminus S$, then $c(S \setminus \{v\})=c(S)$, a contradiction. Hence, $V_1 \setminus \{1\} \subseteq S$. On the other hand, if $w \in V_2 \setminus \{2k\}$, then $w \notin S$. This shows that $S=V_1 \setminus \{1\}$ The other case is similar.

Moreover, it is easy to check that $J_{M_{k,k}}$ and $J_{M_{k-1,k}}$ are unmixed. In Example \ref{E.unmixedNotCM2} we will show that these ideals are not Cohen-Macaulay.
\end{example}

A first nice fact about bipartite graphs with unmixed binomial edge ideal is that they have at least two cut vertices.

\begin{proposition}\label{P.unmixed2leaves}
Let $G$ be a bipartite graph such that $J_G$ is unmixed. Then $G$ has exactly $2$ leaves.
\end{proposition}

\begin{proof}
Let $V(G)=V_1 \sqcup V_2$ be the bipartition of $G$, with $m_1=|V_1| \geq 1$ and $m_2=|V_2| \geq 1$. Assume that $G$ has exactly $h$ leaves, $f_1,\dots,f_h$, in $V_1$ and $k$ leaves, $g_1,\dots,g_k$, in $V_2$. We claim that $S_1=V_1 \setminus \{f_1,\dots,f_h\}$ and $S_2=V_2 \setminus \{g_1,\dots,g_k\}$ are cut sets of $G$. Notice that $c_G(S_1)=|V_2|=m_2$ and $c_G(S_1 \setminus \{v\})<c_G(S_1)$ since the vertex $v$ joins at least two connected components of $G_{\overline S_1}$. By symmetry, the claim is true for $S_2$ and, in particular $c_G(S_2)=|V_1|=m_1$. From the unmixedness of $J_G$ it follows that $\hgt(P_\emptyset(G))=\hgt(P_{S_1}(G))$ and $\hgt(P_\emptyset(G))=\hgt(P_{S_2}(G))$. Thus $n-1=n+|S_1|-c_G(S_1)=n+m_1-h-m_2$ and $n-1=n+|S_2|-c_G(S_2)=n+m_2-k-m_1$. Hence $h=m_1-m_2+1$ and $k=m_2-m_1+1$. The sum of the two equations yields $h+k=2$.
\end{proof}

\begin{remark} \label{unmixed cut vertices}
Assume that $G$ is bipartite and $J_G$ is unmixed. The proof of Proposition \ref{P.unmixed2leaves} implies that:
\begin{itemize}
  \item[(i)] either $h=2$ and $k=0$, i.e., the two leaves are in the same bipartition set and in this case $m_1=m_2+1$, or $h=1$ and $k=1$, i.e., each bipartition set contains exactly one leaf and in this case $m_1=m_2$;
  \item[(ii)] if $G$ has at least $4$ vertices, then the leaves cannot be attached to the same vertex $v$, otherwise $c_G(\{v\}) \geq 3 >2=|\{v\}|+1$, against the unmixedness of $J_G$, see \eqref{Eq.unmixedness}. Hence $G$ has at least two distinct cut vertices, which are the neighbours of the leaves.
\end{itemize}
\end{remark}

\begin{remark}
Notice that Proposition \ref{P.unmixed2leaves} does not hold if $G$ is not bipartite. In fact, there are non-bipartite graphs $G$ with an arbitrary number of leaves and such that $J_G$ is Cohen-Macaulay. For $n \geq 2$ the binomial edge ideal $J_{K_n}$ of the complete graph $K_n$ is Cohen-Macaulay, since it is the ideal of $2$-minors of a generic $(2 \times n)$-matrix (see \cite[Corollary 2.8]{BV88}). Moreover, for $n \geq 3$, $K_n$ has $0$ leaves. Let $W \subseteq [n]$, with $|W|=k \geq 1$. Adding a whisker to a vertex of $W$, the resulting graph $H$ has $1$ leaf and $J_H$ is Cohen-Macaulay by \cite[Theorem 2.7]{RR14}. Applying the same argument to all vertices of $W$, we obtain a graph $H'$ with $k$ leaves such that $J_{H'}$ is Cohen-Macaulay.
\end{remark}

In the remaining part of the section we present a construction, Proposition \ref{Semiconi}, that produces new examples of unmixed binomial edge ideals. It will also be important in the proof of the main theorem.

If $X$ is a subset of $V(G)$, we define the set of neighbours of the elements of $X$, denoted $N_G(X)$, or simply $N(X)$, as the set
\[
N_G(X)=\{y \in V(G) : \{x,y\} \in E(G) \text{ for some } x \in X\}.
\]

\begin{lemma} \label{maximal degree}
Let $G$ be a bipartite graph with bipartition $V_1 \sqcup V_2$, $J_G$ unmixed and let $v_1$ and $v_2$ be the neighbours of the leaves.
\begin{itemize}
\item[{\rm a)}] If $X \subseteq V_1 \setminus \{v_1, v_2\}$, then $N(X)$ is a cut set of $G$ and $|N(X)| \geq |X|$.
\item[{\rm b)}] If $\{v_1,v_2\} \in E(G)$, then $m=|V_1|=|V_2|$ and $v_i$ has degree $m$, for $i=1,2$. Moreover, $v_1$ and $v_2$ are the only cut vertices of $G$.
\end{itemize}
\end{lemma}

\begin{proof} a)  First notice that $N(X)$ is a cut set. In fact, every element of $X$ is isolated in $G_{\overline{N(X)}}$. Let $v \in N(X)$. Then $\deg(v) \geq 2$, since $v_1, v_2 \notin X$. Adding $v$ to $G_{\overline{N(X)}}$, it connects at least a vertex of $X$ with some other connected component.

Now, suppose by contradiction that $|N(X)| < |X|$. Then $G_{\overline{N(X)}}$ has at least $|X|$ isolated vertices and another connected component containing a leaf, because $v_1, v_2 \notin X$. Hence, $c_G(N(X)) \geq |X|+1 >|N(X)|+1$, against the unmixedness of $J_G$.

\medskip
b) Assume that $v_1 \in V_1$. Then $v_2 \in V_2$, since $\{v_1,v_2\} \in E(G)$. By Remark \ref{unmixed cut vertices}(i), it follows that $m=|V_1|=|V_2|$. Define $X=\{w \in V_2 : \{v_1,w \} \notin E(G) \}$ and assume that $X \neq \emptyset$. Since $\{v_1,v_2\} \in E(G)$, $v_2 \notin X$, hence $N(X)$ is a cut set and $|N(X)| \geq |X|$ by a). We claim that the inequality is strict. Assume $|N(X)|=|X|$. Let $f$ be the leaf of $G$ adjacent to $v_1$, then $S=V_2 \setminus (X \cup \{f\})$ is a cut set of $G$ and $|S|=m-|X|-1$. In fact, in $G_{\overline S}$ all vertices of $V_1 \setminus N(X)$ are isolated, except for $v_1$ that is connected only to $f$. Moreover, by definition of $X$, if we add an element of $S$ to $G_{\overline S}$, we join the connected component of $v_1$ with some other connected component of $G_{\overline{S}}$. Thus, $S$ is a cut set and $G_{\overline S}$ consists of at least $|V_1|-|N(X)|-1=m-|X|-1$ isolated vertices, the single edge $\{v_1,f\}$, and the connected component containing the vertices of $X$ and $N(X)$. Hence, $c_G(S) \geq m-|X|+1 > |S|+1$, a contradiction since $J_G$ is unmixed. This shows that $|N(X)|>|X|$.

Now, the vertices of $X$ are isolated in $G_{\overline{N(X)}}$. Moreover, the remaining vertices belong to the same connected component, because, by definition of $X$, $\{v_1,w\} \in E(G)$ for every $w \in V_2 \setminus X$ and all vertices in $V_1 \setminus N(X)$ are adjacent to vertices of $\overline X$. Hence, $c_G(N(X))=|X|+1<|N(X)|+1$, which again contradicts the unmixedness of $J_G$. Hence, $X= \emptyset$ and $v_1$ has degree $m$. In the same way it follows that $v_2$ has degree $m$.

For the last part of the claim, notice that if $v \in V(G) \setminus \{v_1,v_2\}$, the first part implies that every vertex of $G_{\overline{\{v\}}}$ is adjacent to either $v_1$ or $v_2$. Hence, $G_{\overline{\{v\}}}$ is connected and, thus, $v$ is not a cut vertex of $G$.
\end{proof}

\begin{remark}
Let $G$ be a bipartite graph such that $J_G$ is unmixed. If $G$ has exactly two cut vertices, they are not necessarily adjacent. Thus, the converse of the last part of Lemma \ref{maximal degree} b) does not hold. In fact, if $|V_1|=|V_2|+1$, then $v_1$ and $v_2$ belong to the same bipartition set, hence $\{v_1,v_2\} \notin E(G)$. On the other hand, if $|V_1|=|V_2|$, let $G$ be the graph in Figure \ref{A139}. One can check with \texttt{Macaulay2} \cite{Mac2} that the ideal $J_G$ is unmixed, and we notice that the vertices $2$ and $11$ are the only cut vertices, but $\{2,11\} \notin E(G)$.

\begin{figure}[ht!]
\begin{tikzpicture}[scale=1]
\node[label={above:$1$}] (a) at (0,1.5) {};
\node[label={above:$3$}] (b) at (1,1.5) {};
\node[label={above:$5$}] (c) at (2,1.5) {};
\node[label={above:$7$}] (d) at (3,1.5) {};
\node[label={above:$9$}] (e) at (4,1.5) {};
\node[label={above:$11$}] (f) at (5,1.5) {};
\node[label={below:$2$}] (g) at (0,0) {};
\node[label={below:$4$}] (h) at (1,0) {};
\node[label={below:$6$}] (i) at (2,0) {};
\node[label={below:$8$}] (l) at (3,0) {};
\node[label={below:$10$}] (m) at (4,0) {};
\node[label={below:$12$}] (n) at (5,0) {};
\draw (0,1.5) -- (0,0) -- (1,1.5) -- (1,0) -- (2,1.5) -- (2,0) -- (3,1.5) -- (3,0) -- (4,1.5) -- (4,0) -- (5,1.5) -- (5,0)
(2,1.5) -- (0,0)
(2,0) -- (1,1.5) -- (3,0) -- (5,1.5)
(3,0) -- (2,1.5) -- (4,0) -- (3,1.5)
(1,1.5) -- (4,0);
\end{tikzpicture}
\caption{} \label{A139}
\end{figure}
\end{remark}

\begin{proposition} \label{Semiconi}
Let $H$ be a bipartite graph with bipartition $V_1 \sqcup V_2$ and $|V_1|=|V_2|$. Let $v$ and $f$ be two new vertices and let $G$ be the bipartite graph with $V(G)=V(H) \cup \{v,f\}$ and $E(G)=E(H) \cup \{\{v,x\} : x \in V_1 \cup \{f\} \}$. If $J_H$ is unmixed and the neighbours of the leaves of $H$ are adjacent, then $J_G$ is unmixed and
\[
\mathcal M(G) = \{\emptyset, V_1\} \cup \{S \cup \{v\} : S \in \mathcal M(H) \} \cup \{T \subset V_1 : T \in \mathcal M(H)\}.
\]
Moreover, the converse holds if there exists $w \in V_1$ such that $\deg_G(w)=2$.
\end{proposition}

\begin{proof}
Assume that $J_H$ is unmixed and the neighbours of the leaves of $H$ are adjacent. Clearly, $\emptyset, V_1 \in \mc M(G)$. If $S \in \mc M(H)$, then adding $v$ to $G_{\overline{S \cup \{v\}}}$ we join $f$ with some other connected component of $H_{\overline S}$. Moreover, if $w \in S$, adding $w$ to $G_{\overline{S \cup \{v\}}}$ we join at least two connected components of $H_{\overline S}$ (since $S \in \mc M(H)$), which are different components of $G_{\overline{S \cup \{v\}}}$. Finally, let $T \in \mc M(H)$, $T \subset V_1$. By Lemma \ref{maximal degree} b), in $H$ there exists a unique cut vertex $v_2 \in V_2$ and $N_H(v_2)=V_1$. Hence, adding $w \in T$ to $G_{\overline T}$, we join at least two components since $N_G(v)=V_1 \cup \{f\}$ and $T \in \mc M(H)$.

Conversely, let $S \in \mc M(G)$ and suppose first that $v \in S$. Then $G_{\overline S} = H_{\overline{S \setminus \{v\}}} \sqcup \{f\}$ and this implies that $S \setminus \{v\}$ is a cut set of $H$, since every element of $S \setminus \{v\}$ has to join some connected components that only contain vertices of $H_{\overline{S \setminus \{v\}}}$. Therefore $c_G(S)=c_H(S \setminus \{v\})+1=|S|+1$.

Suppose now that $v \notin S$. Let $w$ be the leaf of $H$ adjacent to $v_2$, that is also adjacent to $v$ in $G$. First of all, notice that $S \subset V_1$. Indeed, in $G_{\overline S}$ every vertex of $V_1 \setminus S$ is in the same connected component of $v$. Thus, a vertex of $V_2$ cannot join different connected components. Since $w$ is adjacent only to $v$ and $v_2$, if $w \in S$, then $v$ and $v_2$ cannot be in the same connected component of $G_{\overline S}$. This means that $V_1 \subset S$, because all the vertices of $V_1$ are adjacent to $v$ and $v_2$, by Lemma \ref{maximal degree} b). Thus $S=V_1$ and $c_G(S)=|V_2|+1=|S|+1$. Hence, we may assume that $w \notin S$. We claim that, in this case, $S \in \mc M(H)$. In fact, it is clear that $v_2,w,v$ and $f$ are in the same connected component $C$ of $G_{\overline S}$, which also contains all vertices of $V_1 \setminus S$, since they are adjacent to $v$. Then, the connected components of $G_{\overline S}$ and $H_{\overline S}$ are the same except for $C$, that in $H_{\overline S}$ is $C_{\overline{\{v,f\}}} \neq \emptyset$. Therefore, if $x \in S$ joins two connected components of $G_{\overline S}$, it also joins the same connected components of $H_{\overline S}$ (or $C_{\overline{\{v,f\}}}$, if it joins $C$), hence $S$ is a cut set of $H$. Moreover, $c_G(S)=c_H(S)=|S|+1$.

Conversely, assume that $J_G$ is unmixed and let $S \in \mc M(H)$. Notice that $w$ is a leaf of $H$, hence $w \notin S$, by Remark \ref{R.leafNotCutSet}. We prove that $T=S \cup \{v\}$ is a cut set of $G$. As before, $G_{\overline T}=H_{\overline S} \cup \{f\}$. Thus the elements of $S$ join different connected components also in $G_{\overline T}$ and $v$ connects the isolated vertex $f$ with the connected component of $w$. Hence, $T \in \mc M(G)$ and $c_H(S)=c_G(T)-1=|T|+1-1=|S|+1$.

Finally, let $v_i$ be the cut vertex of $H$ in $V_i$ for $i=1,2$. Since $\{v, v_1\} \in E(G)$, it follows, from Lemma \ref{maximal degree} b), that $\{v_1,v_2\} \in E(G)$. Then $v_1$ and $v_2$ are adjacent also in $H$.
\end{proof}

In Figure \ref{F.semicono}, we show an example of the above construction. The ideal $J_G$ is unmixed by Proposition \ref{Semiconi}, since $H=M_{4,4}$ and $J_H$ is unmixed by Example \ref{E.unmixedNotCM1}. Moreover, it will follow from Example \ref{E.unmixedNotCM2} and Proposition \ref{P.preconoconnesso} that $J_G$ is not Cohen-Macaulay.

\begin{figure}[ht!]
\begin{tikzpicture}
\node[lightgray] (a) at (0,1.5) {};
\node[lightgray] (b) at (1,1.5) {};
\node[lightgray] (c) at (2,1.5) {};
\node[lightgray] (d) at (3,1.5) {};
\draw[lightgray] (0,1.5) -- (0,0) -- (1,1.5) -- (1,0) -- (2,1.5) -- (2,0) -- (3,1.5) -- (3,0)
(2,1.5) -- (0,0) -- (3,1.5) -- (1,0)
(1,1.5) -- (2,0);
\node (e) at (0,0) {};
\node (f) at (1,0) {};
\node (g) at (2,0) {};
\node (h) at (3,0) {};
\node[label={right:$f$}] (j) at (4,0) {};
\node[label={right:$v$}] (k) at (4,1.5) {};
\draw (0,0) -- (4,1.5) -- (1,0)
(2,0) -- (4,1.5) -- (3,0)
(4,0) -- (4,1.5);
\end{tikzpicture}
\caption{} \label{F.semicono}
\end{figure}

In Proposition \ref{Semiconi}, the existence of a vertex $w \in V_1$ such that $\deg_G(w)=2$ means that $w$ is a leaf of $H$. This is not true in general, see for instance the graph $M_{k,k}$ in Example \ref{E.unmixedNotCM1} for $k \geq 4$. However, if $J_H$ is unmixed, this always holds:

\begin{corollary} \label{C.unmixedSemicone}
Let $H$ be a bipartite graph with bipartition $V_1 \sqcup V_2$, $|V_1|=|V_2|$ and such that $J_H$ is unmixed. Let $G$ be the graph in Proposition {\rm \ref{Semiconi}}. Then $J_G$ is unmixed if and only if the neighbours of the leaves of $H$ are adjacent.
\end{corollary}

\begin{example}
The graph $H$ of Figure \ref{A139} is such that $J_H$ is unmixed, but the two cut vertices $2$ and $11$ are not adjacent. The graph in Figure \ref{F.counterexample2} is the graph $G$ obtained from $H$ with the construction in Proposition \ref{Semiconi}. According to Corollary \ref{C.unmixedSemicone}, $J_G$ is not unmixed: in fact $S=N(11)=\{8,10,12\}$ is a cut set and $c_G(S)=3 \neq |S|+1$.

\begin{figure}[ht!]
\begin{tikzpicture}
\node[lightgray,label={above:$1$}] (a) at (0,1.5) {};
\node[lightgray,label={above:$3$}] (b) at (1,1.5) {};
\node[lightgray,label={above:$5$}] (c) at (2,1.5) {};
\node[lightgray,label={above:$7$}] (d) at (3,1.5) {};
\node[lightgray,label={above:$9$}] (e) at (4,1.5) {};
\node[lightgray,label={above:$11$}] (f) at (5,1.5) {};
\draw[lightgray] (0,1.5) -- (0,0) -- (1,1.5) -- (1,0) -- (2,1.5) -- (2,0) -- (3,1.5) -- (3,0) -- (4,1.5) -- (4,0) -- (5,1.5) -- (5,0)
(2,1.5) -- (0,0)
(2,0) -- (1,1.5) -- (3,0) -- (5,1.5)
(3,0) -- (2,1.5) -- (4,0) -- (3,1.5)
(1,1.5) -- (4,0);
\node[label={below:$2$}] (h) at (0,0) {};
\node[label={below:$4$}] (i) at (1,0) {};
\node[label={below:$6$}] (l) at (2,0) {};
\node[label={below:$8$}] (m) at (3,0) {};
\node[label={below:$10$}] (n) at (4,0) {};
\node[label={below:$12$}] (o) at (5,0) {};
\node[label={above:$13$}] (g) at (6,1.5) {};
\node[label={below:$14$}] (p) at (6,0) {};
\draw (0,0) -- (6,1.5) -- (1,0) -- (6,1.5) -- (2,0) -- (6,1.5) -- (3,0) -- (6,1.5) -- (4,0) -- (6,1.5) -- (5,0) -- (6,1.5) -- (6,0);
\end{tikzpicture}
\caption{} \label{F.counterexample2}
\end{figure}
\end{example}

\section{Basic blocks} \label{S.basicBlocks}

In this section we study the basic blocks $F_m$ of our classification, proving that $J_{F_m}$ is Cohen-Macaulay.

\medskip
In what follows we will use several times the following argument.

\begin{remark}\label{R.splittingPrimaryDecomposition}
Let $G$ be a graph, $v$ be a vertex of $G$, $H'=G \setminus \{v\}$ and assume that $\mc M(H')=\{S \setminus \{v\} : S \in \mc M(G), v \in S\}$, in particular $v$ is a cut vertex of $G$ since $\emptyset \in \mc M(H')$. Let $J_G=\bigcap_{S \in \mathcal M(G)} P_S(G)$ be the primary decomposition of $J_G$ and set $A=\bigcap_{S \in \mathcal M(G), v \notin S} P_S(G)$ and $B=\bigcap_{S \in \mathcal M(G), v \in S} P_S(G)$. Then $J_G = A \cap B$ and we have the short exact sequence
\begin{equation} \label{Eq.shortExactSeq}
0 \longrightarrow R/J_G \longrightarrow R/A \oplus R/B \longrightarrow R/(A+B) \longrightarrow 0.
\end{equation}
Notice that
\begin{itemize}
  \item[i)] $A=J_H$, where $H$ is the graph obtained from $G$ by adding all possible edges between the vertices of $N_G(v)$. In other words, $V(H)=V(G)$ and $E(H)=E(G) \cup \{\{k,\ell\} : k,\ell \in N_G(v), k \neq \ell\}$. In fact, notice that $v \notin S$ for every $S \in \mc M(H)$ by Remark \ref{R.leafNotCutSet} and all cut sets of $G$ not containing $v$ are cut sets of $H$ as well. Thus, $\mathcal M(H)= \{S \in \mathcal M(G) : v \notin S\}$. Moreover, for every $S \in \mc M(H)$, the connected components of $G_{\overline S}$ and $H_{\overline S}$ are the same, except for the component containing $v$, which is $G_i$ in $G_{\overline S}$ and $H_i$ in $H_{\overline S}$. Nevertheless, $\widetilde G_i = \widetilde H_i$, hence $P_S(G)=P_S(H)$ for every $S \in \mathcal M(H)$.
  \item[ii)] $B=(x_v,y_v)+J_{H'}$, where $H'=G \setminus \{v\}$. In fact, if $S \in \mathcal M(G)$ with $v \in S$, then $S \setminus \{v\} \in \mc M(H')$ by assumption and we have that $P_S(G)=(x_v,y_v)+P_{S \setminus \{v\}}(H')$. Thus,
      \[
      B=(x_v,y_v)+\bigcap_{S \in \mathcal M(G), v \in S} P_{S \setminus \{v\}}(H')= (x_v,y_v)+\bigcap_{T \in \mathcal M(H')} P_T(H')=(x_v,y_v)+J_{H'}.
      \]
  \item[iii)] $A+B=(x_v,y_v)+J_{H''}$, where $H''=H \setminus \{v\}$.
\end{itemize}
\end{remark}

We now describe a new family of Cohen-Macaulay binomial edge ideals associated with non-bipartite graphs, which will be useful in what follows.
Let $K_n$ be the complete graph on the vertex set $[n]$ and $W=\{v_1,\dots,v_r\} \subseteq [n]$. Let $H$ be the graph obtained from $K_n$ by attaching, for every $i=1,\dots,r$, a complete graph $K_{h_i}$ to $K_n$ in such a way that $V(K_n) \cap V(K_{h_i}) = \{v_1,\dots,v_i\}$, for some $h_i>i$. We say that the graph $H$ is obtained by \textit{adding a fan to $K_n$ on the set $S$}. For example, Figure \ref{F.fan} shows the result of adding a fan to $K_6$ on a set $S$ of three vertices.

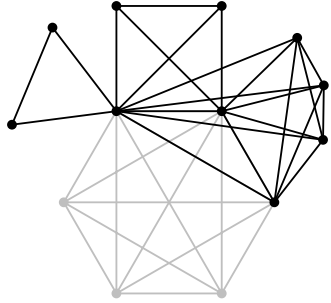
\begin{figure}[ht!]
\begin{tikzpicture}[scale=1.4,line join=round]
\draw[lightgray] (0.,0.)-- (1.,0.)-- (1.5,0.8660254037844387)-- (1.,1.7320508075688776)-- (0.,1.7320508075688779)-- (-0.5,0.8660254037844395)-- (0.,0.)-- (0.,1.7320508075688779)-- (1.5,0.8660254037844387)-- (0,0)-- (1.,1.7320508075688776)-- (1,0)-- (-0.5,0.8660254037844395)-- (1.,1.7320508075688776);
\draw[lightgray] (-0.5,0.8660254037844395)-- (1.5,0.8660254037844387);
\draw[lightgray] (1,0)-- (0.,1.7320508075688779);
\draw (1.,1.7320508075688776)-- (0.,1.7320508075688779);
\draw (1.5,0.8660254037844387)-- (1.,1.7320508075688776);
\draw (0.,1.7320508075688779)-- (0.,2.7320508075688776);
\draw (0.,2.7320508075688776)-- (1.,2.7320508075688776);
\draw (1.,2.7320508075688776)-- (1.,1.7320508075688776);
\draw (0.,1.7320508075688779)-- (1.,2.7320508075688776);
\draw (0.,2.7320508075688776)-- (1.,1.7320508075688776);
\draw (-0.9917272617078157,1.6036878908619228)-- (0.,1.7320508075688779);
\draw (-0.607029177625997,2.526730351479947)-- (-0.9917272617078157,1.6036878908619228);
\draw (-0.607029177625997,2.526730351479947)-- (0.,1.7320508075688779);
\draw (1.7169050039063332,2.4292216728336706)-- (1.,1.7320508075688776);
\draw (1.9694433781154632,1.9773659857134)-- (1.,1.7320508075688776);
\draw (1.9622201820508556,1.4597782947182076)-- (1.,1.7320508075688776);
\draw (1.5,0.8660254037844387)-- (1.9622201820508556,1.4597782947182076);
\draw (1.9694433781154632,1.9773659857134)-- (1.5,0.8660254037844387);
\draw (1.7169050039063332,2.4292216728336706)-- (1.5,0.8660254037844387);
\draw (1.7169050039063332,2.4292216728336706)-- (1.9622201820508556,1.4597782947182076);
\draw (1.9694433781154632,1.9773659857134)-- (1.9622201820508556,1.4597782947182076);
\draw (1.9694433781154632,1.9773659857134)-- (1.7169050039063332,2.4292216728336706);
\draw (1.7169050039063332,2.4292216728336706)-- (0.,1.7320508075688779);
\draw (1.9694433781154632,1.9773659857134)-- (0.,1.7320508075688779);
\draw (0.,1.7320508075688779)-- (1.9622201820508556,1.4597782947182076);
\draw (0.,1.7320508075688779)-- (1.5,0.8660254037844387);
\node[lightgray] (a) at (0.,0.) {};
\node[lightgray] (b) at (1.,0.) {};
\node (c) at (1.5,0.8660254037844387) {};
\node (d) at (1.,1.7320508075688776) {};
\node (e) at (0.,1.7320508075688779) {};
\node[lightgray] (f) at (-0.5,0.8660254037844395) {};
\node (g) at (0.,2.7320508075688776) {};
\node (h) at (1.,2.7320508075688776) {};
\node (i) at (-0.607029177625997,2.526730351479947) {};
\node (j) at (-0.9917272617078157,1.6036878908619228) {};
\node (k) at (1.7169050039063332,2.4292216728336706) {};
\node (l) at (1.9694433781154632,1.9773659857134) {};
\node (m) at (1.9622201820508556,1.4597782947182076) {};
\end{tikzpicture}
\caption{Adding a fan to $K_6$ on three vertices}\label{F.fan}
\end{figure}

\begin{lemma} \label{L.fanOnCompleteGraphCM}
Let $K_n$ be the complete graph on $[n]$ and $W_1 \sqcup \cdots \sqcup W_k$ be a partition of a subset $W \subseteq [n]$. Let $G$ be the graph obtained from $K_n$ by adding a fan on each set $W_i$. Then $J_G$ is Cohen-Macaulay.
\end{lemma}

\begin{proof}
First we show that $J_G$ is unmixed. For every $i=1,\dots,k$, set $W_i=\{v_{i,1},\dots,v_{i,r_i}\}$ and $\mathcal M_i=\{\emptyset\} \cup \{\{v_{i,1},\dots,v_{i,h}\} : 1 \leq h \leq r_i\}$. We claim that
\begin{equation}\label{Eq.cutSetsFan}
\mathcal M(G) = \{T_1 \cup \cdots \cup T_k : T_i \in \mathcal M_i, T_1 \cup \cdots \cup T_k \subsetneq [n] \}.
\end{equation}

Let $T=T_1 \cup \cdots \cup T_k \neq \emptyset$, with $T_i \in \mathcal M_i$ for $i=1,\dots,k$, and $T \subsetneq [n]$. Let $v \in T$. Then $v \in T_j$ for some $j$, say $v=v_{j,\ell}$, with $1 \leq \ell \leq r_j$. Hence, if we add $v$ to the graph $G_{\overline T}$, it joins the connected component containing $K_n \setminus T$ (which is non-empty since $T \subsetneq [n]$) with $K_{h_{j,\ell}} \setminus T$, where $V(K_{h_{j,\ell}}) \cap V(G) = \{v_{j,1},\dots,v_{j,\ell}\}$. This shows that $c_G(T)>c_G(T \setminus \{v\})$ for every $v \in T$, thus $T \in \mathcal M(G)$.

Conversely, let $T \in \mathcal M(G)$. First notice that $T \neq [n]$, since $c_G([n])=c_G([n] \setminus \{v_{i,r_i}\})$ for every $i$. Moreover, $T$ does not contain any vertex $v \in V(G) \setminus \bigcup_{i=1}^k W_i$, otherwise $v$ belongs to exactly one maximal clique of $G$, see Remark \ref{R.leafNotCutSet}. Then $c_G(T)=c_G(T \setminus \{v\})$. Hence $T \subseteq \bigcup_{i=1}^k W_i$ and $T \subsetneq [n]$. Let $T = \bigcup_{i=1}^k T_i$, where $T_i \subseteq W_i$. We want to show that, if $v_{i,j} \in T_i \subseteq T$, then $v_{i,h} \in T$ for every $1 \leq h < j$. Assume $v_{i,h} \notin T_i$ for some $h < j$. Then $c_G(T)=c_G(T \setminus \{v_{i,j}\})$ because all maximal cliques of $G$ containing $v_{i,j}$ contain $v_{i,h}$ as well, since $h<j$. This shows that $T_i \in \mathcal M_i$ for every $i$.

Finally, for every $T \in \mathcal M(G)$, since $G_{\overline T}$ consists of  $|T|$ connected components that are complete graphs ($K_{h_{j,\ell}} \setminus T$ for every $j=1,\dots,k$ and $\ell=1,\dots,|T_j|$) and a graph obtained from $K_n \setminus T$ by adding a fan on each $W_i \setminus T_i$, it follows that $c_G(T)=|T|+1$. This means that $J_G$ is unmixed and $\dim(R/J_G) = |V(G)|+1$.

In order to prove that $J_G$ is Cohen-Macaulay, we proceed by induction on $k \geq 1$ and $|S_k| \geq 1$. Let $k=1$ and set $W_1=\{1,\dots,r\}$. If $|W_1|=1$, then the claim follows by \cite[Theorem 2.7]{RR14}. Assume that $|W_1|=r \geq 2$ and the claim true for $r-1$. Notice that $G=\mathrm{cone}(1,G_1 \sqcup G_2)$, where $G_1 \cong K_{h_1-1}$ (graph isomorphism) and $G_2$ is the graph obtained from $K_n \setminus \{1\}$ by adding a fan on the clique $\{2,\dots,r\}$. We know that $J_{G_1}$ is Cohen-Macaulay by \cite[Corollary 2.8]{BV88} and $J_{G_2}$ is Cohen-Macaulay by induction. Hence, the claim follows by \cite[Theorem 3.8]{RR14}.

Now, let $k \geq 2$ and assume the claim true for $k-1$. Again, if $|W_k|=1$, the claim follows by induction and by \cite[Theorem 2.7]{RR14}. Assume that $|W_k|=r_k \geq 2$ and the claim true for $r_k-1$. For simplicity, let $W_k=\{1,\dots,r_k\}$.
Let $J_G=\bigcap_{S \in \mathcal M(G)} P_S(G)$ be the primary decomposition of $J_G$ and set $A=\bigcap_{S \in \mathcal M(G), 1 \notin S} P_S(G)$ and $B=\bigcap_{S \in \mathcal M(G), 1 \in S} P_S(G)$. Then $J_G=A \cap B$.

By Remark \ref{R.splittingPrimaryDecomposition}, $A=J_H$, where $H$ is a complete graph on the vertices of $\{1\} \cup N_G(1)$ to which we add a fan on the cliques $W_1,\dots,W_{k-1}$. Hence $R/A$ is Cohen-Macaulay by induction on $k$ and $\depth(R/A)=|V(G)|+1$.

Notice that $H'=G \setminus \{1\}$ is the disjoint union of a complete graph and a graph $K'$, which is obtained by adding a fan to $K_n \setminus \{1\} \cong K_{n-1}$ on the cliques $W_1,\dots,W_{k-1}$ and $W_k \setminus \{1\}$. From \eqref{Eq.cutSetsFan}, it follows that $\mc M(H')=\{S \setminus \{1\} : S \in \mc M(G), 1 \in S\}$, thus $B=(x_1,y_1)+J_{H'}$ by Remark \ref{R.splittingPrimaryDecomposition}. By induction on $|W_k|$, $J_{K'}$ is Cohen-Macaulay, hence $J_{H'}$ is Cohen-Macaulay since it is the sum of Cohen-Macaulay ideals on disjoint sets of variables. In particular, $\depth(R/B)=|V(H')|+2=|V(G)|+1$ (it follows from the formula for the dimension \cite[Corollary 3.4]{HHHKR10}).

Finally, by Remark \ref{R.splittingPrimaryDecomposition}, $A+B=(x_1,y_1)+J_{H''}$, where $H''=H \setminus \{1\}$. Hence $R/(A+B)$ is Cohen-Macaulay by induction on $k$ and $\depth(R/(A+B))=|V(G)|$.

The Depth Lemma \cite[Lemma 3.1.4]{V94} applied to the short exact sequence \eqref{Eq.shortExactSeq} yields $\depth(R/J_G)=|V(G)|+1$. The claim follows from the first part, since $\dim(R/J_G)=|V(G)|+1$.
\end{proof}

Notice that the graphs produced by Lemma \ref{L.fanOnCompleteGraphCM} are not generalized block graphs (see \cite{KM15}) nor closed graphs if $k \geq 2$ (studied in \cite{EHH11}). Hence they form a new family of non-biparite graphs whose binomial edge ideal is Cohen-Macaulay.

Now we prove that the binomial edge ideals of the graphs $F_m$ (see Figure \ref{F.exampleF_m}) are Cohen-Macaulay. The graphs $F_m$ are the basic blocks in our classification, Theorem \ref{T.CHARACTERIZATION}.

Recall that, for every $m \geq 1$, if $n=2m$, $F_m$ is the graph on the vertex set $[n]$ and with edge set
\[
E(F_m) = \left\{ \{2i,2j-1\} : i=1,\dots,m, j=i,\dots,m \right\}.
\]

Notice that $F_m$, with $m \geq 2$, can be obtained from $F_{m-1}$ using the construction of Proposition \ref{Semiconi}.

\begin{proposition}\label{P.type1CM}
For every $m \geq 1$, $J_{F_m}$ is Cohen-Macaulay.
\end{proposition}

\begin{proof}
First we show that $J_{F_m}$ is unmixed.
We proceed by induction on $m \geq 1$. If $m=1$, then $J_{F_1}$ is a principal ideal, hence it is prime and unmixed of height $1$. Let $m \geq 2$ and assume the claim true for $m-1$. Then $F_m$ is obtained from $F_{m-1}$ by adding the vertices $n-1$ and $n$ and connecting $n-1$ to the vertices $2,4,\dots,n$. Since $J_{F_{m-1}}$ is unmixed by induction and $\{2,n-3\} \in E(F_{m-1})$, by Proposition \ref{Semiconi}, it follows that $J_{F_m}$ is unmixed and
\begin{equation}\label{Eq.cutSetsFm}
\mc M(F_m)=\{\emptyset\} \cup \{ \{2,4,\dots,2i\} : 1 \leq i \leq m-1 \} \cup \{\{n-1\} \cup S: S \in \mc M(F_{m-1}) \}.
\end{equation}

Now we prove that $J_{F_m}$ is Cohen-Macaulay by induction on $m \geq 1$. The graphs $F_1$ and $F_2$ are paths, hence the ideals $J_{F_1}$ and $J_{F_2}$ are complete intersections, by \cite[Corollary 1.2]{EHH11}, thus Cohen-Macaulay.

Let $m \geq 3$ and assume that $J_{F_{m-1}}$ is Cohen-Macaulay. Let $J_{F_m}=\bigcap_{S \in \mathcal M(F_m)} P_S(F_m)$ be the primary decomposition of $J_{F_m}$ and define $A=\bigcap_{S \in \mathcal M(F_m), n-1 \notin S} P_S(F_m)$ and $B=\bigcap_{S \in \mathcal M(F_m), n-1 \in S} P_S(F_m)$. Then $J_{F_m} = A \cap B$.

By Remark \ref{R.splittingPrimaryDecomposition}, $A=J_H$, where $H$ is obtained by adding a fan to the complete graph with vertex set $N_{F_m}(n-1)=\{2,4,\dots,n\}$ on the set $N_{F_m}(n-1)$, hence it is Cohen-Macaulay by Lemma \ref{L.fanOnCompleteGraphCM} and $\depth(R/A)=n+1$.

Since $F_m \setminus \{n-1\}=F_{m-1} \sqcup \{n\}$, by \eqref{Eq.cutSetsFm}, $\mc M(F_{m-1} \sqcup \{n\})=\{S \setminus \{n-1\} : S \in \mc M(F_m), n-1 \in S\}$. Thus, $B=(x_{n-1},y_{n-1})+J_{F_{m-1} \sqcup \{n\}}=(x_{n-1},y_{n-1})+J_{F_{m-1}}$, hence it is Cohen-Macaulay by induction and $\depth(R/B)=n+1$.

Finally, $A+B=(x_{n-1},y_{n-1})+J_{H''}$, where $H''=H \setminus \{n-1\}$, which is Cohen-Macaulay again by Lemma \ref{L.fanOnCompleteGraphCM} and $\depth(R/(A+B))=n$.

The Depth Lemma applied to the exact sequence \eqref{Eq.shortExactSeq} yields $\depth(R/J_{F_m})=n+1$. Moreover, since $J_{F_m}$ is unmixed, it follows that $\dim(R/J_{F_m})=n+1$ and, therefore, $J_{F_m}$ is Cohen-Macaulay.
\end{proof}

\section{Gluing graphs: operations \texorpdfstring{$\ast$}{*} and \texorpdfstring{$\circ$}{°}} \label{S.gluing}

In this section we consider two operations that, together with the graphs $F_m$, are the main ingredients of Theorem \ref{T.CHARACTERIZATION}. Given two (not necessarily bipartite) graphs $G_1$ and $G_2$, we glue them to obtain a new graph $G$. If $G_1$ and $G_2$ are bipartite, both constructions preserve the Cohen-Macaulayness of the associated binomial edge ideal. The first operation is a particular case of the one studied by Rauf and Rinaldo in \cite[Section 2]{RR14}.

\begin{definition}
For $i=1,2$, let $G_i$ be a graph with at least one leaf $f_i$. We define the graph $G=(G_1,f_1) \ast (G_2,f_2)$ obtained by identifying $f_1$ and $f_2$ (see Figure \ref{F.example2aNonBip}). If it is not important to specify the vertices $f_i$ or it is clear from the context, we simply write $G_1 \ast G_2$.
\end{definition}

\begin{figure}[ht!]
\begin{subfigure}[c]{0.25\textwidth}
\centering
\begin{tikzpicture}[scale=0.9]
\draw (0.,2.)-- (0.,0.);
\draw (1.,1.)-- (2.,1.);
\draw (-1.,1.)-- (0.,0.);
\draw (0.,0.)-- (1.,1.);
\draw (0.,2.)-- (1.,1.);
\draw (0.,2.)-- (-1.,1.);
\draw (0.,2.)-- (0.,0.);
\node (a) at (0.,0.) {};
\node (b) at (0.,2.) {};
\node (c) at (-1.,1.) {};
\node (d) at (1.,1.) {};
\node[label={right:{\small $f_1$}}] (e) at (2.,1.) {};
\end{tikzpicture}
\caption{A graph $G_1$}
\end{subfigure}
\begin{subfigure}[c]{0.25\textwidth}
\centering
\begin{tikzpicture}[scale=0.9]
\draw (2.,1.)-- (3.,1.);
\draw (3.,1.)-- (4.,2.);
\draw (4.,2.)-- (5.,1.);
\draw (3.,1.)-- (4.,0.);
\draw (4.,0.)-- (5.,1.);
\draw (5.,1.)-- (6.,1.);
\node[label={left:{\small $f_2$}}] (e) at (2.,1.) {};
\node (f) at (3.,1.) {};
\node (g) at (4.,2.) {};
\node (h) at (5.,1.) {};
\node (i) at (4.,0.) {};
\node (j) at (6.,1.) {};
\end{tikzpicture}
\caption{A graph $G_2$}
\end{subfigure}
\begin{subfigure}[c]{0.45\textwidth}
\centering
\begin{tikzpicture}[scale=0.9]
\draw (0.,2.)-- (0.,0.);
\draw (1.,1.)-- (2.,1.);
\draw (2.,1.)-- (3.,1.);
\draw (3.,1.)-- (4.,2.);
\draw (4.,2.)-- (5.,1.);
\draw (3.,1.)-- (4.,0.);
\draw (4.,0.)-- (5.,1.);
\draw (5.,1.)-- (6.,1.);
\draw (-1.,1.)-- (0.,0.);
\draw (0.,0.)-- (1.,1.);
\draw (0.,2.)-- (1.,1.);
\draw (0.,2.)-- (-1.,1.);
\draw (0.,2.)-- (0.,0.);
\node (a) at (0.,0.) {};
\node (b) at (0.,2.) {};
\node (c) at (-1.,1.) {};
\node (d) at (1.,1.) {};
\node[label={[label distance=-2mm]270:{\small $f_1=f_2$}}] (e) at (2.,1.) {};
\node (f) at (3.,1.) {};
\node (g) at (4.,2.) {};
\node (h) at (5.,1.) {};
\node (i) at (4.,0.) {};
\node (j) at (6.,1.) {};
\end{tikzpicture}
\caption{The graph $(G_1,f_1) \ast (G_2,f_2)$}
\end{subfigure}
\caption{} \label{F.example2aNonBip}
\end{figure}

In the next Theorem we recall some results about the operation $\ast$, see \cite[Lemma 2.3, Proposition 2.6, Theorem 2.7]{RR14}.

\begin{theorem} \label{T.type2aCM}
For $i=1,2$, consider a graph $G_i$ with at least one leaf $f_i$ and $G=(G_1,f_1) \ast (G_2,f_2)$. Let $v_1$ and $v_2$ be the neighbours of the leaves and let $v$ be the vertex obtained by identifying $f_1$ and $f_2$. If
\begin{gather*}
\mathcal A = \{S_1 \cup S_2 : S_i \in \mathcal M(G_i), \, i=1,2\} \text{ and}\\
\mathcal B = \{S_1 \cup S_2 \cup \{v\} : S_i \in \mathcal M(G_i) \text{ and } v_i \notin S_i, \, i=1,2 \},
\end{gather*}
the following properties hold:
\begin{itemize}
\item[{\rm a)}] $\mc M(G) = \mc A \cup \mc B$;
\item[{\rm b)}] $J_G$ is unmixed if and only if $J_{G_1}$ and $J_{G_2}$ are unmixed;
\item[{\rm c)}] $J_G$ is Cohen-Macaulay if and only if $J_{G_1}$ and $J_{G_2}$ are Cohen-Macaulay.
\end{itemize}
\end{theorem}

We now introduce the second operation.

\begin{definition}
For $i=1,2$, let $G_i$ be a graph with at least one leaf $f_i$, $v_i$ its neighbour and assume $\deg_{G_i}(v_i) \geq 3$. We define $G = (G_1,f_1) \circ (G_2,f_2)$ to be the graph obtained from $G_1$ and $G_2$ by removing the leaves $f_1,f_2$ and identifying $v_1$ and $v_2$ (see Figure \ref{F.2bNotUnmixed}). If it is not important to specify the leaves $f_i$ or it is clear from the context, then we simply write $G_1 \circ G_2$.
\end{definition}

We denote by $v$ the vertex of $G$ resulting from the identification of $v_1$ and $v_2$ and, with abuse of notation, we write $V(G_1) \cap V(G_2) = \{v\}$.

Notice that, if $\deg_{G_i}(v_i)=2$ for $i=1,2$, then $(G_1,f_1) \circ (G_2,f_2) = (G_1 \setminus \{f_1\},v_1) \ast (G_2 \setminus \{f_2\},v_2)$. On the other hand, we do not allow $\deg_{G_1}(v_1)=2$ and $\deg_{G_2}(v_2) \geq 3$ (or vice versa), since in this case the operation $\circ$ does not preserve unmixedness, see Remark \ref{R.2bRemark} (ii).

\begin{remark} \label{no converse for 2b}
Unlike the operation $\ast$ (cf. Theorem \ref{T.type2aCM}), if one of $J_{G_1}$ and $J_{G_2}$ is not Cohen-Macaulay, then $J_{G_1 \circ G_2}$ may not be unmixed, even if $G_1$ and $G_2$ are bipartite. For example, let $G_1$ and $G_2$ be the graphs in Figure \ref{F.2bNotUnmixed1} and \ref{F.2bNotUnmixed2}. Then $J_{G_1 \circ G_2}$ is not unmixed even if $J_{G_1}=J_{F_4}$ is Cohen-Macaulay (by Proposition \ref{P.type1CM}) and $J_{G_2}=J_{M_{4,4}}$ is unmixed (by Example \ref{E.unmixedNotCM1}). In fact, $S=\{5,7,8,10,12\} \in \mathcal M(G)$, but $c_G(S) = 5 \neq |S|+1$.

\begin{figure}[ht!]
\begin{subfigure}[c]{0.3\textwidth}
\centering
\begin{tikzpicture}
\node (a) at (0,1.5) {};
\node (b) at (1,1.5) {};
\node (c) at (2,1.5) {};
\node[label={right:{\small $v_1$}}] (d) at (3,1.5) {};
\node (e) at (0,0) {};
\node (f) at (1,0) {};
\node (g) at (2,0) {};
\node[label={right:{\small $f_1$}}] (h) at (3,0) {};
\draw (0,1.5) -- (0,0) -- (1,1.5) -- (1,0) -- (2,1.5) -- (2,0) -- (3,1.5) -- (3,0)
(2,1.5) -- (0,0) -- (3,1.5) -- (1,0);
\end{tikzpicture}
\caption{The graph $G_1$} \label{F.2bNotUnmixed1}
\end{subfigure}
\begin{subfigure}[c]{0.3\textwidth}
\centering
\begin{tikzpicture}
\node[label={left:{\small $v_2$}}] (a) at (0,1.5) {};
\node (b) at (1,1.5) {};
\node (c) at (2,1.5) {};
\node (d) at (3,1.5) {};
\node[label={left:{\small $f_2$}}] (e) at (0,0) {};
\node (f) at (1,0) {};
\node (g) at (2,0) {};
\node (h) at (3,0) {};
\draw (0,0) -- (0,1.5) -- (1,0) -- (1,1.5) -- (2,0) -- (2,1.5) -- (3,0) -- (3,1.5)
(2,0) -- (0,1.5) -- (3,0) -- (1,1.5)
(1,0) -- (2,1.5);
\end{tikzpicture}
\caption{The graph $G_2$} \label{F.2bNotUnmixed2}
\end{subfigure}
\begin{subfigure}[c]{0.38\textwidth}
\centering
\begin{tikzpicture}
\node[label={above:{\small $1$}}] (a) at (0,1.5) {};
\node[label={above:{\small $3$}}] (b) at (1,1.5) {};
\node[label={above:{\small $5$}}] (c) at (2,1.5) {};
\node[label={above:{\small $7$}}] (d) at (3,1.5) {};
\node[label={below:{\small $2$}}] (e) at (0,0) {};
\node[label={below:{\small $4$}}] (f) at (1,0) {};
\node[label={below:{\small $6$}}] (g) at (2,0) {};
\node[label={below:{\small $8$}}] (h) at (4,0) {};
\node[label={above:{\small $9$}}] (i) at (4,1.5) {};
\node[label={below:{\small $10$}}] (j) at (5,0) {};
\node[label={above:{\small $11$}}] (k) at (5,1.5) {};
\node[label={below:{\small $12$}}] (l) at (6,0) {};
\node[label={above:{\small $13$}}] (m) at (6,1.5) {};
\draw (0,1.5) -- (0,0) -- (1,1.5) -- (1,0) -- (2,1.5) -- (2,0) -- (3,1.5) -- (4,0) -- (4,1.5) -- (5,0) -- (5,1.5) -- (6,0) -- (6,1.5)
(2,1.5) -- (0,0) -- (3,1.5) -- (1,0)
(5,0) -- (3,1.5) -- (6,0) -- (4,1.5)
(4,0) -- (5,1.5);
\end{tikzpicture}
\caption{The graph $G=G_1 \circ G_2$}
\end{subfigure}
\caption{} \label{F.2bNotUnmixed}
\end{figure}
\end{remark}

We describe the structure of the cut sets of $G_1 \circ G_2$ under some extra assumption on $G_1$ and $G_2$. In this case, $\circ$ preserves unmixedness.

\begin{theorem}\label{T.type2bBipartiteUnmixed}
Let $G=G_1 \circ G_2$ and set $V(G_1) \cap V(G_2)=\{v\}$, where $\deg_{G_i}(v) \geq 3$ for $i=1,2$. If for $i=1,2$ there exists $u_i \in N_{G_i}(v)$ with $\deg_{G_i}(u_i)=2$, then
\begin{equation}\label{Eq.cutSets2b}
\mathcal{M}(G)=\mc A \cup \mc B,
\end{equation}
where
\begin{gather*}
\mathcal A = \{S_1 \cup S_2 : S_i \in \mathcal M(G_i), i=1,2, v \notin S_1 \cup S_2\} \text{ and}\\
\mathcal B = \{S_1 \cup S_2 : S_i \in \mathcal M(G_i), i=1,2, S_1 \cap S_2=\{v\} \}.
\end{gather*}

If $J_{G_1}$ and $J_{G_2}$ are unmixed and for $i=1,2$ there exists $u_i \in N_{G_i}(v)$ with $\deg_{G_i}(u_i)=2$, then $J_G$ is unmixed. The converse holds if $G$ is bipartite. In particular, if $G$ is bipartite and $J_G$ is unmixed, the cut sets of $G$ are described in \eqref{Eq.cutSets2b}.
\end{theorem}

\begin{proof}
Let $S=S_1 \cup S_2 \subset V(G)$, where $S_1=S \cap V(G_1)$ and $S_2=S \cap V(G_2)$. Notice that
\begin{eqnarray}\label{Eq.connComponents2b}
c_G(S) &=& c_{G_1}(S_1)+c_{G_2}(S_2)-1, \text{ if } v \notin S, \label{Eq.connComp2b2}\\
c_G(S) &=& c_{G_1}(S_1)+c_{G_2}(S_2)-2, \text{ if } v \in S. \label{Eq.connComp2b1}
\end{eqnarray}

In fact, if $v \notin S$, the connected components of $G_{\overline S}$ are those of $(G_1)_{\overline S_1}$ and $(G_2)_{\overline S_2}$, where the component containing $v$ is counted once. On the other hand, if $v \in S$, clearly $v \in S_1 \cap S_2$ and the connected components of $G_{\overline S}$ are those of $(G_1)_{\overline S_1}$ and $(G_2)_{\overline S_2}$, except for the two leaves $f_1$ and $f_2$.

In order to prove \eqref{Eq.cutSets2b}, we show the two inclusions.

$\subseteq$: Let $S \in \mathcal{M}(G)$ and define $S_1$ and $S_2$ as before. Suppose by contradiction that $S_1 \notin \mc M(G_1)$, i.e., there exists $w \in S_1$ such that $c_{G_1}(S_1)=c_{G_1}(S_1 \setminus \{w\})$. If $v \notin S$, then by \eqref{Eq.connComp2b2}
\[
c_G(S \setminus \{w\})=c_{G_1}(S_1 \setminus \{w\}) + c_{G_2}(S_2) - 1 = c_{G_1}(S_1) + c_{G_2}(S_2) - 1 = c_G(S),
\]
a contradiction. On the other hand, if $v \in S$ and $w \neq v$, by \eqref{Eq.connComp2b1} we have
\[
c_G(S \setminus \{w\})=c_{G_1}(S_1 \setminus \{w\}) + c_{G_2}(S_2) - 2 = c_{G_1}(S_1) + c_{G_2}(S_2) - 2 = c_G(S),
\]
again a contradiction. We show that the case $w=v$ cannot occur. In fact, by assumption, there exists $u_1 \in N_{G_1}(v)$ such that $\deg_{G_1}(u_1)=2$. Since $v \in S$, we have that $c_G(S)=c_G(S \setminus \{u_1\})$, hence $u_1 \notin S$. Thus $c_{G_1}(S_1)>c_{G_1}(S_1 \setminus \{v\})$, because by adding $v$ to $(G_1)_{\overline S_1}$, we join the connected component of $u_1$ and the isolated vertex $f_1$, which is a leaf in $G_1$. Hence $w \neq v$. The same argument also shows that $S_2 \in \mc M(G_2)$.

$\supseteq$: Let $S=S_1 \cup S_2$, with $S_i \in \mathcal{M}(G_i)$, for $i=1,2$. Assume first $S_1 \cap S_2=\{v\}$. By the equalities \eqref{Eq.connComp2b2} and \eqref{Eq.connComp2b1} we have
\[
c_G(S \setminus \{v\})=c_{G_1}(S_1 \setminus \{v\})+c_{G_2}(S_2 \setminus \{v\})-1 \leq c_{G_1}(S_1)+c_{G_2}(S_2)-3=c_G(S)-1<c_G(S).
\]
Let $w \in S$, $w \neq v$. Without loss of generality, we may assume $w \in S_1$. Then
\[
c_G(S \setminus \{w\})=c_{G_1}(S_1 \setminus \{w\})+c_{G_2}(S_2)-2 \leq c_{G_1}(S_1)+c_{G_2}(S_2)-3=c_G(S)-1<c_G(S).
\]

Assume now that $v \notin S_1 \cup S_2$. Let $w \in S$, and without loss of generality $w \in S_1$. Then
\[
c_G(S \setminus \{w\})=c_{G_1}(S_1 \setminus \{w\})+c_{G_2}(S_2)-1 \leq c_{G_1}(S_1)+c_{G_2}(S_2)-2=c_G(S)-1<c_G(S).
\]

Let now $J_{G_1}$ and $J_{G_2}$ be unmixed and for $i=1,2$ there exists $u_i \in N_{G_i}(v)$ with $\deg_{G_i}(u_i)=2$. By the last assumption, the cut sets of $G$ are described in \eqref{Eq.cutSets2b}. Let $S \in \mathcal{M}(G)$ and $S_i=S \cap V(G_i)$ for $i=1,2$. Thus, by \eqref{Eq.connComp2b2} and \eqref{Eq.connComp2b1},
\begin{itemize}
\item[(i)] if $v \notin S$, $c_G(S)=c_{G_1}(S_1)+c_{G_2}(S_2)-1=|S_1|+1+|S_2|+1-1=|S_1|+|S_2|+1=|S|+1$,
\item[(ii)] if $v \in S$, $c_G(S)=c_{G_1}(S_1)+c_{G_2}(S_2)-2=|S_1|+1+|S_2|+1-2=|S_1|+|S_2|=|S|+1$.
\end{itemize}
It follows that $J_G$ is unmixed.

Conversely, let $J_G$ be unmixed and $G$ bipartite. If $S$ is a cut set of $G_1$, then it is also a cut set of $G$ and clearly $c_{G_1}(S)=c_G(S)$; therefore $J_{G_1}$ is unmixed and the same holds for $J_{G_2}$. By Proposition \ref{P.unmixed2leaves}, the graphs $G$, $G_1$ and $G_2$ have exactly two leaves. Let $f_i$ be the leaf of $G_i$ adjacent to $v$ and $g_i$ be the other leaf of $G_i$. Thus, $g_1$ and $g_2$ are the leaves of $G$.

By symmetry, it is enough to prove that there exists $u_1 \in N_{G_1}(v)$ such that $\deg_{G_1}(u_1)=2$. For $i=1,2$, let $V(G_i)=V_i \cup W_i$ and assume $|V_1| \leq |W_1|$. By Remark \ref{unmixed cut vertices}, we have one of the following two cases:
\begin{itemize}
  \item[a)] if $|V_1|=|W_1|$, we may assume $f_1 \in W_1$ and $g_1 \in V_1$. Set $S=(W_1 \setminus \{f_1\}) \cup \{v\}$. Hence, $c_{G_1}(S)=|V_1|=|W_1|=|S|$.
  \item[b)] If $|W_1|=|V_1|+1$, then $f_1,g_1 \in W_1$. Hence $v \in V_1$. Set $S=(W_1 \setminus \{f_1,g_1\}) \cup \{v\}$. Thus, $c_{G_1}(S)=|V_1|=|W_1|-1=|S|$.
\end{itemize}

First suppose $|V(G_2)|$ even and assume $f_2 \in W_2$. Hence, $v,g_2 \in V_2$ and $T=V_2 \setminus \{g_2\}$ is a cut set of $G_2$.

Now, let $|V(G_2)|$ be odd and assume $f_2 \in W_2$. Hence, $g_2 \in W_2$, $v \in V_2$ and $|W_2|=|V_2|+1$. Then $T=V_2$ is a cut set of $G_2$.

In both cases, notice that $S \cup T$ is not a cut set of $G$, since $S \cap T = \{v\}$ and, by \eqref{Eq.connComp2b1},
\[
c_G(S \cup T)=c_{G_1}(S)+c_{G_2}(T)-2=|S|+|T|-1=|S \cup T|,
\]
which contradicts the unmixedness of $J_G$. Let $u \in S \cup T$ such that $c_G((S \cup T) \setminus \{u\})=c_G(S \cup T)=|S \cup T|$.

We show that $u \in S$ and $u \neq v$. If $u \notin S$, then $u \in T$ and $u \neq v$. By \eqref{Eq.connComp2b1},
\[
c_G((S \cup T) \setminus \{u\})=c_{G_1}(S)+c_{G_2}(T \setminus \{u\})-2<|S|+c_{G_2}(T)-2=|S|+|T|-1=|S \cup T|=c_G(S \cup T),
\]
against our assumption (the inequality holds since $T$ is a cut set of $G_2$ and the second equality follows from the unmixedness of $J_{G_2}$). Thus, $u \in S$. Moreover, in both cases $c_{G_1}(S \setminus \{v\})=c_{G_1}(S)=|S|$ (since $v$ is a leaf of $(G_1)_{\overline{S \setminus \{v\}}}$) and, by \eqref{Eq.connComp2b2},
\[
c_G((S \cup T) \setminus \{v\}) \!=\! c_{G_1}(S \setminus \{v\})+c_{G_2}(T \setminus \{v\})-1 \!=\! |S|+|T|-|N_{G_2}(v)|+2-1 \!<\! |S|+|T|-1 \!=\! |S \cup T| \!=\! c_G(S \cup T),
\]
where the inequality holds since $\deg_{G_2}(v) \geq 3$. This contradicts our assumption, thus $u \neq v$.

We conclude that $u \in S \setminus \{v\}$. Since $u \neq f_1,g_1$, we have $\deg_{G_1}(u) \geq 2$. On the other hand, since $c_G((S \cup T) \setminus \{u\})=c_G(S \cup T)$, it follows that $u \in N_{G_1}(v)$ and $\deg_{G_1}(u)=2$.
\end{proof}

\begin{corollary} \label{T.2bUnmixed}
Let $G=F_{m_1} \circ \cdots \circ F_{m_k}$, where $m_i \geq 3$ for $i=1,\dots,k$. Then $J_G$ is unmixed.
\end{corollary}

\begin{proof}
Set $G_1=F_{m_1} \circ \cdots \circ F_{m_{k-1}}$, $G_2=F_{m_k}$ and let $v$ be the only vertex of $V(G_1) \cap V(G_2)$. We proceed by induction on $k \geq 2$. If $k=2$, the claim follows by Theorem \ref{T.type2bBipartiteUnmixed}, because $J_{G_1}$ and $J_{G_2}$ are unmixed by Proposition \ref{P.type1CM} and for $i=1,2$, there exists $u_i \in N_{G_i}(v)$ such that $\deg_{G_i}(u_i)=2$, by definition of $F_{m_i}$.

Now let $k>2$ and assume the claim true for $k-1$. By induction, $J_{G_1}$ is unmixed. Since $m_{k-1} \geq 3$, there exists $u_1 \in N_{G_1}(v)$ such that $\deg_{G_1}(u_1)=2$. The claim follows again by Theorem \ref{T.type2bBipartiteUnmixed}.
\end{proof}

\begin{remark}\label{R.2bRemark}
In Corollary \ref{T.2bUnmixed} the condition $m_i \geq 3$, for $i=2,\dots,k-1$, cannot be omitted. For instance, the binomial edge ideal $J_{F_3 \circ F_2 \circ F_3}$ is not unmixed: in fact $S=\{3,5,6,8\}$ is a cut set and $c_{F_3 \circ F_2 \circ F_3}(S)=4 \neq |S|+1$, see Figure \ref{F.counterexample2b}.

\begin{figure}[ht!]
\begin{tikzpicture}
\node[label={above:{\small $1$}}] (a) at (0,1.5) {};
\node[label={above:{\small $3$}}] (b) at (1,1.5) {};
\node[label={above:{\small $5$}}] (c) at (2,1.5) {};
\node[label={above:{\small $7$}}] (d) at (3,1.5) {};
\node[label={above:{\small $9$}}] (e) at (4,1.5) {};
\node[label={below:{\small $2$}}] (f) at (0,0) {};
\node[label={below:{\small $4$}}] (g) at (1,0) {};
\node[label={below:{\small $6$}}] (h) at (2,0) {};
\node[label={below:{\small $8$}}] (i) at (3,0) {};
\node[label={below:{\small $10$}}] (l) at (4,0) {};
\draw (0,1.5) -- (0,0) -- (1,1.5) -- (1,0) -- (2,1.5) -- (2,0) -- (3,1.5) -- (3,0) -- (4,1.5) -- (4,0)
(2,1.5) -- (0,0)
(4,1.5) -- (2,0);
\end{tikzpicture}
\caption{The graph $F_3 \circ F_2 \circ F_3$} \label{F.counterexample2b}
\end{figure}
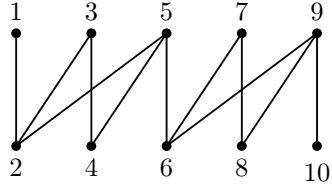
\noindent On the other hand, we may allow $m_1=m_k=2$, since, in this case, the graph $G=F_{m_1} \circ \cdots \circ F_{m_k}=F_1 \ast F_{m_2} \circ \cdots \circ F_{m_{k-1}} \ast F_1$. Hence, $J_G$ is unmixed by Theorem \ref{T.type2aCM} and Corollary \ref{T.2bUnmixed}.
\end{remark}

Let $n \geq 3$, $W_1 \sqcup \cdots \sqcup W_k$ be a partition of a subset of $[n]$ and $W_i=\{v_{i,1},\dots,v_{i,r_i}\}$ for some $r_i \geq 1$ and $i=1,\dots,k$. Let $E$ be the graph obtained from $K_n$ by adding a fan on each set $W_i$ in such a way that we attach a complete graph $K_{h+1}$ to $K_n$, with $V(K_n) \cap V(K_{h+1}) = \{v_{i,1},\dots,v_{i,h}\}$, for $i=1,\dots,k$ and $h=1,\dots,r_i$, see Figure \ref{F.graphE} (cf. Figure \ref{F.fan}). By Lemma \ref{L.fanOnCompleteGraphCM}, $J_E$ is Cohen-Macaulay.

\begin{figure}[ht!]
\begin{tikzpicture}[scale=1.4,line join=round]
\draw[lightgray] (0.,0.)-- (1.,0.)-- (1.5,0.8660254037844387)-- (1.,1.7320508075688776)-- (0.,1.7320508075688779)-- (-0.5,0.8660254037844395)-- (0.,0.)-- (0.,1.7320508075688779)-- (1.5,0.8660254037844387)-- (0,0)-- (1.,1.7320508075688776)-- (1,0)-- (-0.5,0.8660254037844395)-- (1.,1.7320508075688776);
\draw[lightgray] (-0.5,0.8660254037844395)-- (1.5,0.8660254037844387);
\draw[lightgray] (1,0)-- (0.,1.7320508075688779);
\draw (1.,1.7320508075688776)-- (0.,1.7320508075688779);
\draw (1.5,0.8660254037844387)-- (1.,1.7320508075688776);
\draw (0.5,2.7320508075688776)-- (1.,1.7320508075688776);
\draw (0.,1.7320508075688779)-- (0.5,2.7320508075688776);
\draw (-0.607029177625997,2.526730351479947)-- (0.,1.7320508075688779);
\draw (1.7169050039063332,2.4292216728336706)-- (1.,1.7320508075688776);
\draw (1.7169050039063332,2.4292216728336706)-- (1.5,0.8660254037844387);
\draw (1.7169050039063332,2.4292216728336706)-- (0.,1.7320508075688779);
\draw (0.,1.7320508075688779)-- (1.5,0.8660254037844387);
\node[lightgray] (a) at (0.,0.) {};
\node[lightgray] (b) at (1.,0.) {};
\node (c) at (1.5,0.8660254037844387) {};
\node (d) at (1.,1.7320508075688776) {};
\node (e) at (0.,1.7320508075688779) {};
\node[lightgray] (f) at (-0.5,0.8660254037844395) {};
\node (g) at (0.5,2.7320508075688776) {};
\node (i) at (-0.607029177625997,2.526730351479947) {};
\node (k) at (1.7169050039063332,2.4292216728336706) {};
\end{tikzpicture}
\caption{The graph $E$}\label{F.graphE}
\end{figure}
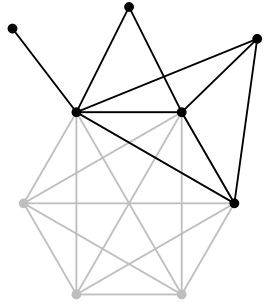

\begin{lemma} \label{L.2bAndFanUnmixed}
Let $G=F_{m_1} \circ \cdots \circ F_{m_k} \circ E$, where $E$ is the graph defined above, $m_i \geq 3$ for every $i=2,\dots,k$ and $V(F_{m_1} \circ \cdots \circ F_{m_k}) \cap V(E)=\{v\}$. Assume that $v \in W_1$ and $|W_1| \geq 2$. Then $J_G$ is unmixed.
\end{lemma}

\begin{proof}
Set $G_1=F_{m_1} \circ \cdots \circ F_{m_k}$ and $G_2=E$. Then $J_{G_1}$ is unmixed by Corollary \ref{T.2bUnmixed} and $J_{G_2}$ is Cohen-Macaulay by Lemma \ref{L.fanOnCompleteGraphCM}, hence it is unmixed.

Notice that, since $m_k \geq 3$, there exists $u_1 \in N_{G_1}(v)$ such that $\deg_{G_1}(u_1)=2$. Moreover, since $|W_1| \geq 2$ and by definition of $G_2=E$, we attach $K_3$ to $K_n$ in such a way that $|V(K_n) \cap V(K_3)|=2$ and $v \in V(K_n) \cap V(K_3)$. Thus, there exists $u_2 \in K_3$, hence $u_2 \in N_{G_2}(v)$, such that $\deg_{G_2}(u_2)=2$. The statement follows by Theorem \ref{T.type2bBipartiteUnmixed}.
\end{proof}

In Lemma \ref{L.2bAndFanUnmixed} we assume $|W_1| \geq 2$, since this is the only case we need in the following theorem. Moreover, in the next statement the case $F=E$ is useful to prove that the binomial edge ideal associated with $F_{m_1} \circ \cdots \circ F_{m_k} \circ F_n$ is Cohen-Macaulay.

\begin{theorem}\label{T.type2bCMbipartite}
Let $G=F_{m_1} \circ \cdots \circ F_{m_k} \circ F$, where $m_i \geq 3$ for every $i=1,\dots,k$ and $F=F_n$ for some $n \geq 3$ or $F=E$ is the same graph of Lemma {\rm \ref{L.2bAndFanUnmixed}}. Then $J_G$ is Cohen-Macaulay.
\end{theorem}

\begin{proof}
Let $V(F_{m_1} \circ \dots \circ F_{m_k}) \cap V(F)= \{w\}$ and call $f_k$ and $f$ the leaves that we remove from $F_{m_1} \circ \cdots \circ F_{m_k}$ and $F$. Let $J_G=\bigcap_{S \in \mathcal M(G)} P_S(G)$ be the primary decomposition of $J_G$ and set $A=\bigcap_{S \in \mathcal M(G), w \notin S} P_S(G)$ and $B=\bigcap_{S \in \mathcal M(G), w \in S} P_S(G)$.

We proceed by induction on $k \geq 1$. First assume $k=1$ and, for simplicity, let $m=m_1$.
By Remark \ref{R.splittingPrimaryDecomposition}, the ideal $A$ is the binomial edge ideal of the graph $H$, obtained by adding a fan to the complete graph with vertex set $\{w\} \cup N_G(w)$ on the sets $N_{F_m}(w) \setminus \{f_k\}$ and $N_F(w) \setminus \{f\}$. Hence $R/A$ is Cohen-Macaulay and $\depth(R/A)=|V(G)|+1$ by Lemma \ref{L.fanOnCompleteGraphCM}.

Notice that $G \setminus \{w\}=(F_m \setminus \{w, f_k\}) \sqcup (F \setminus \{w, f\})$. By Theorem \ref{T.type2bBipartiteUnmixed} and Remark \ref{R.splittingPrimaryDecomposition}, $B=(x_w,y_w)+J_{F_m \setminus \{w, f_k\}} + J_{F \setminus \{w, f\}}$, where $F_m \setminus \{w,f\} \cong F_{m-1}$. Moreover, if $F=E$, $E \setminus \{w,f\}$ is of the same form as $E$, otherwise $F=F_n$ and $F_n \setminus \{w,f\} \cong F_{n-1}$. In any case, $J_{F_m \setminus \{w, f_k\}}$ and $J_{F \setminus \{w, f\}}$ are Cohen-Macaulay (by Lemma \ref{L.fanOnCompleteGraphCM} and Proposition \ref{P.type1CM}), hence $B$ is Cohen-Macaulay since it is the sum of Cohen-Macaulay ideals on disjoint sets of variables. In particular, it follows from the formula for the dimension \cite[Corollary 3.4]{HHHKR10} that $\depth(R/B)=|V(F_{m-1})|+1+|V(F \setminus \{w,f\})|+1=|V(G)|+1$.

Finally, $A+B=(x_w,y_w)+J_{H''}$, where $H''=H \setminus \{w\}$ is the binomial edge ideal of the graph obtained by adding a fan to the complete graph with vertex set $N_G(w)$ on the sets $N_{F_m}(w) \setminus \{f_k\}$ and $N_F(w) \setminus \{f\}$. Hence $R/(A+B)$ is Cohen-Macaulay and $\depth(R/(A+B))=|V(G)|$ by Lemma \ref{L.fanOnCompleteGraphCM}.

The Depth Lemma applied to the short exact sequence \eqref{Eq.shortExactSeq} yields $\depth(R/J_G)=|V(G)|+1$. The claim follows by Lemma \ref{L.2bAndFanUnmixed} (resp. Corollary \ref{T.2bUnmixed}), since $\dim(R/J_G)=|V(G)|+1$.

Now let $k>1$ and assume the claim true for $k-1$. By Remark \ref{R.splittingPrimaryDecomposition}, the ideal $A$ is the binomial edge ideal of the graph $H=F_1 \circ \cdots \circ F_{m_{k-1}} \circ F'$, where $F'$ is obtained by adding a fan to the complete graph with vertex set $\{w\} \cup N_G(w)$ on the sets $N_{F_{m_k}}(w) \setminus \{f_k\}$ and $N_F(w) \setminus \{f\}$. Notice that, since $m_k \geq 3$, $|N_{F_{m_k}}(w) \setminus \{f_k\}| \geq 2$ and we are in the assumption of Lemma \ref{L.2bAndFanUnmixed}. Hence, $R/A$ is Cohen-Macaulay by induction and $\depth(R/A)=|V(G)|+1$.

Similarly to the case $k=1$, the ideal $B$ equals $(x_w,y_w)+J_{(F_{m_1} \circ \cdots \circ F_{m_k}) \setminus \{w, f_k\}} + J_{F \setminus \{w, f\}}$, where $(F_{m_1} \circ \cdots \circ F_{m_k}) \setminus \{w, f_k\} \cong F_{m_1} \circ \cdots \circ F_{m_k-1}$ and $J_{F_{m_1} \circ \cdots \circ F_{m_k-1}}$ is Cohen-Macaulay by induction (notice that, if $m_k=3$, then $F_{m_1} \circ \cdots \circ F_{m_k}=F_{m_1} \circ \cdots \circ F_{m_{k-1}} \ast F_1$ and the corresponding binomial edge ideal is Cohen-Macaulay by induction and Theorem \ref{T.type2aCM}). Moreover, if $F=E$, then $E \setminus \{w\}$ is of the same form as $E$, otherwise $F=F_n$ and $F_n \setminus \{w,f\} \cong F_{n-1}$. Thus $J_{F \setminus \{w,f\}}$ is Cohen-Macaulay (by Lemma \ref{L.fanOnCompleteGraphCM} and Proposition \ref{P.type1CM}), hence $B$ is Cohen-Macaulay since it is the sum of Cohen-Macaulay ideals on disjoint sets of variables. In particular, $\depth(R/B)=|V(F_{m-1})|+1+|V(F \setminus \{w,f\})|+1=|V(G)|+1$ (it follows from the formula for the dimension \cite[Corollary 3.4]{HHHKR10}).

Finally, $A+B=(x_w,y_w)+J_{H''}$, where $H''=H \setminus \{w\}$ (again, since $m_k \geq 3$, we have $|N_{F_{m_k}}(w) \setminus \{f_k\}| \geq 2$). Hence $R/(A+B)$ is Cohen-Macaulay by induction and $\depth(R/(A+B))=|V(G)|$.

The Depth Lemma applied to the short exact sequence \eqref{Eq.shortExactSeq} yields $\depth(R/J_G)=|V(G)|+1$. Notice that, if $F=E$, the ideal $J_G$ is unmixed by Lemma \ref{L.2bAndFanUnmixed}, whereas, if $F=F_n$, it is unmixed by Corollary \ref{T.2bUnmixed}. This implies that $\dim(R/J_G)=|V(G)|+1$ and the claim follows.
\end{proof}

\section{The dual graph of binomial edge ideals} \label{S.dualGraph}

In this section we study the dual graph of binomial edge ideals. This is one of the main tools to prove that, if $G$ is bipartite and $J_G$ is Cohen-Macaulay, then $G$ can be obtained recursively via a sequence of operations $\ast$ and $\circ$ on a finite set of graphs of the form $F_m$, Theorem \ref{T.CHARACTERIZATION} c).

Let $I$ be an ideal in a polynomial ring $A=K[x_1,\dots,x_n]$ and let $\mathfrak p_1, \dots, \mathfrak p_r$ be the minimal prime ideals of $I$. Following \cite{BV15}, the \textit{dual graph} $\mathcal D(I)$ of $I$ is a graph with vertex set $\{1,\dots,r\}$ and edge set
\[
\{\{i,j\} : \hgt(\mathfrak p_i+\mathfrak p_j)-1=\hgt(\mathfrak p_i)=\hgt(\mathfrak p_j)=\hgt(I)\}.
\]

Notice that, if $\mc D(I)$ is connected, then $I$ is unmixed. In \cite{H62}, Hartshorne proved that if $A/I$ is Cohen-Macaulay, then $\mathcal{D}(I)$ is connected. We will show that this is indeed an equivalence for binomial edge ideals of bipartite graphs. Nevertheless, this does not hold when $G$ is not bipartite, see Remark \ref{connessononCM}.

To ease the notation, we denote by $\mathcal D(G)$ the dual graph of the binomial edge ideal $J_G$ of a graph $G$. Moreover, we denote by $P_S(G)$ or $P_S$ both the minimal primes of $J_G$ and the vertices of $\mathcal D(G)$.

\begin{remark}\label{connessononCM}
The dual graph of the non-bipartite graph $G$ in Figure \ref{F.unmixedDualConnected} is connected, see Figure \ref{F.DualUnmixedDualConnected}, but using \texttt{Macaulay2} \cite{Mac2} one can check that $J_G$ is not Cohen-Macaulay.

\begin{figure}[ht!]
\begin{subfigure}[c]{0.4\textwidth}
\centering
\begin{tikzpicture}[scale=0.9]
\draw (0.,0.)-- (2.,0.);
\draw (2.,0.)-- (2.,2.);
\draw (2.,2.)-- (0.,2.);
\draw (0.,2.)-- (0.,0.);
\draw (2.,0.)-- (4.,0.);
\draw (2.,2.)-- (4.,2.);
\draw (0.,0.)-- (1.,1.);
\draw (1.,1.)-- (0.,2.);
\draw (0.,2.)-- (-1.,1.);
\draw (-1.,1.)-- (0.,0.);
\node[label={below:$4$}] (a) at (0.,0.) {};
\node[label={below:$5$}] (b) at (2.,0.) {};
\node[label={above:$2$}] (c) at (2.,2.) {};
\node[label={above:$3$}] (d) at (0.,2.) {};
\node[label={below:$6$}] (e) at (4.,0.) {};
\node[label={above:$1$}] (f) at (4.,2.) {};
\node[label={right:$7$}] (g) at (1.,1.) {};
\node[label={left:$8$}] (h) at (-1.,1.) {};
\end{tikzpicture}
\caption{The graph $G$} \label{F.unmixedDualConnected}
\end{subfigure}
\begin{subfigure}[c]{0.5\textwidth}
\centering
\begin{tikzpicture}[scale=1.1]
\draw (0.,0.)-- (0.7071067811865476,0.7071067811865475);
\draw (0.,0.)-- (0.7071067811865476,-0.7071067811865475);
\draw (0.7071067811865476,0.7071067811865475)-- (1.4142135623730951,0.);
\draw (1.4142135623730951,0.)-- (0.7071067811865476,-0.7071067811865475);
\draw (0.7071067811865476,0.7071067811865475)-- (1.4142135623730951,1.414213562373095);
\draw (1.4142135623730951,1.414213562373095)-- (2.621320343559643,0.7071067811865476);
\draw (2.621320343559643,0.7071067811865476)-- (3.3284271247461907,0.);
\draw (0.7071067811865476,-0.7071067811865475)-- (1.4142135623730951,-1.414213562373095);
\draw (1.4142135623730951,1.414213562373095)-- (1.4142135623730951,0.);
\draw (1.4142135623730951,0.)-- (1.4142135623730951,-1.414213562373095);
\draw (1.4142135623730951,-1.414213562373095)-- (2.621320343559643,-0.7071067811865474);
\draw (2.621320343559643,-0.7071067811865474)-- (2.621320343559643,0.7071067811865476);
\draw (2.621320343559643,-0.7071067811865474)-- (3.3284271247461907,0.);
\node[label={left:$P_\emptyset$}] (a) at (0.,0.) {};
\node[label={above left:$P_{\{5\}}$}] (b) at (0.7071067811865476,0.7071067811865475) {};
\node[label={below left:$P_{\{2\}}$}] (c) at (0.7071067811865476,-0.7071067811865475) {};
\node[label={right:$P_{\{2,5\}}$}] (d) at (1.4142135623730951,0.) {};
\node[label={[label distance=-2mm]90:$P_{\{3,5\}}$}] (e) at (1.4142135623730951,1.414213562373095) {};
\node[label={[label distance=-2mm]270:$P_{\{2,4\}}$}] (f) at (1.4142135623730951,-1.414213562373095) {};
\node[label={above right:$P_{\{3,4,5\}}$}] (g) at (2.621320343559643,0.7071067811865476) {};
\node[label={below right:$P_{\{2,3,4\}}$}] (h) at (2.621320343559643,-0.7071067811865474) {};
\node[label={right:$P_{\{3,4\}}$}] (i) at (3.3284271247461907,0.) {};
\end{tikzpicture}
\caption{The dual graph of $G$} \label{F.DualUnmixedDualConnected}
\end{subfigure}
\caption{}
\end{figure}
\end{remark}

We now describe the edges of the dual graph of $J_G$, when $J_G$ is unmixed. This result holds for non-bipartite graphs as well.

\begin{theorem} \label{Dual graph}
Let $G$ be a graph such that $J_G$ is unmixed and let $S,T \in \mc M(G)$, with $|T| \geq |S|$. Denote by $P_S$ the minimal primes of $J_G$. Then the following properties hold:
\begin{itemize}
\item[{\rm a)}] if $|T \setminus S| > 1$, then $\{P_S,P_T\}$ is not an edge of $\mathcal{D}(G)$;
\item[{\rm b)}] if $|T \setminus S| = 1$ and $S \subset T$, then $\{P_S,P_T\}$ is an edge of $\mathcal{D}(G)$;
\item[{\rm c)}] if $T \setminus S = \{t\}$ and $S \nsubseteq T$, then $\{P_S,P_T\}$ is an edge of $\mathcal{D}(G)$ if and only if $t$ is not a cut vertex of $G_{\overline S}$.
\end{itemize}
\end{theorem}

\begin{proof} Let $E_1, E_2, \dots, E_{c(S)}$ be the connected components of $G_{\overline S}$.

a) \ Let $v,w \in T \setminus S$. Then $P_S + P_T \supseteq P_S + (x_{v},x_{w},y_{v},y_{w})$. If $E_j$ and $E_k$ are the connected components of $G_{\overline S}$ containing $v$ and $w$ respectively (possibly $j=k$), it follows that
$$
P_S + (x_{v},x_{w},y_{v},y_{w}) = \left( \, \bigcup_{i \in S \cup \{v,w\}} \{x_i, y_i\}, J_{\widetilde{E}_1}, \dots, J_{(\widetilde{E}_j)_{\overline{\{v\}}}}, \dots, J_{(\widetilde{E}_k)_{\overline{\{w\}}}} \dots, J_{\widetilde{E}_{c(S)}} \, \right).
$$
Thus, $\hgt(P_S+P_T) \geq \hgt(P_S + (x_{v},x_{w},y_{v},y_{w})) = \hgt(P_S)+4-2=\hgt(P_S)+2$. Hence, $\{P_S,P_T\}$ is not an edge of $\mathcal{D}(G)$.

b) \ Let $T \setminus S=\{t\}$ and let $E_j$ be the connected component of $G_{\overline S}$ containing $t$. Then
$$P_S + P_T= \left( \, \bigcup_{i \in S} \ \{x_i, y_i\}, (x_t,y_t), J_{\widetilde{E}_1}, \dots, J_{(\widetilde{E}_{j})_{\overline{\{t\}}}}, \dots, J_{\widetilde{E}_{c(S)}} \, \right).
$$
Thus, $\hgt(P_S+P_T)=\hgt(P_S)+2-1=\hgt(P_S)+1$. Hence, $\{P_S,P_T\}$ is an edge of $\mathcal{D}(G)$.

c) \ Let $G_1,G_2,\dots,G_r$ be the connected components of $G_{\overline{S \cap T}}$. Let also $S \setminus T=\{s\}$, $T \setminus S=\{t\}$ and assume that $s \in G_j$ and $t \in G_k$. Since $S$, $T \in \mc M(G)$, it follows that $s$ and $t$ are cut vertices of $G_j$ and $G_k$, respectively.

If $j \neq k$, then $t$ is a cut vertex of $G_{\overline S}$. Moreover, if $V(G_j)=V(E_1 \cup \dots \cup E_h \cup\{s\})$, where $h \geq 2$ and $G_k=E_{h+1}$, then
$$
P_S + P_T= \left( \, \bigcup_{i \in S \cup \{t\}} \{x_i, y_i\}, J_{(\widetilde{G}_j)_{\overline{\{s\}}}}, J_{(\widetilde{E}_{h+1})_{\overline{\{t\}}}}, J_{\widetilde{E}_{h+2}}, \dots, J_{\widetilde{E}_{c(S)}} \, \right).
$$
It follows that $\hgt(P_S+P_T)=\hgt(P_S)+2+|V(G_j)|-2-\sum_{i=1}^h (|V(E_i)|-1) -1=\hgt(P_S)+2+1-2+h-1=\hgt(P_S)+h > \hgt(P_S)+1$. Thus, $\{P_S,P_T\}$ is not an edge of $\mathcal{D}(G)$.

Assume now that $j = k$ and let $j=1$ for simplicity. Denote by $H_1, \dots, H_i$ the connected components of $(G_1)_{\overline{\{s\}}}$ and by $K_1, \dots, K_i$ the connected components of $(G_1)_{\overline{\{t\}}}$ (note that the number of components is the same because $S,T \in \mc M(G)$ and $J_G$ is unmixed). Suppose also that $t \in H_1$ and $s \in K_1$. If there exists $v \in H_p \cap K_q$ with $p,q \neq 1$, then, since $v \in H_p$, there exists a path from $v$ to $s$ that does not involve $t$. This is a contradiction because $v \in K_q$ and $s \in K_1$. Hence, $K_q \subseteq H_1$ and $H_p \subseteq K_1$ for all $p,q=2,\dots,i$. In particular, the connected components of $G_{\overline{S \cup T}}$ are $H_2, \dots, H_i, K_2, \dots, K_i, G_2, \dots, G_r$ and the connected components of $H_1 \cap K_1$, if it is not empty.

Suppose first that $H_1 \cap K_1 = \emptyset$. Hence, $V(H_1)=V(K_2 \cup \dots \cup K_i \cup \{t\})$ and $V(K_1)=V(H_2 \cup \dots \cup H_i \cup \{s\})$. If $i \geq 3$, then $t$ is a cut vertex of $H_1$, hence a cut vertex of $G_{\overline S}$. It follows that
\begin{gather*}
P_S = \left( \, \bigcup_{h \in S} \ \{x_h, y_h\}, J_{\widetilde{H}_1}, J_{\widetilde{H}_2}, \dots, J_{\widetilde{H}_i}, J_{\widetilde{G}_2}, \dots, J_{\widetilde{G}_r} \, \right) \text{ and} \\
P_S + P_T= \left( \, \bigcup_{h \in S \cup \{t\}} \{x_h, y_h\}, J_{(\widetilde{H}_1)_{\overline{\{t\}}}}, J_{(\widetilde{K}_1)_{\overline{\{s\}}}}, J_{\widetilde{G}_2}, \dots, J_{\widetilde{G}_r} \, \right).
\end{gather*}
Therefore, $\hgt(P_S+P_T)=\hgt(P_S)+2-1-\sum_{h=2}^i (|V(H_h)|-1)+|V(K_1)|-2=\hgt(P_S)+1-\sum_{h=2}^i |V(H_h)|+(i-1) +(\sum_{h=2}^i |V(H_h)| +1)-2=\hgt(P_S)+i-1 > \hgt(P_S) +1$, since $i \geq 3$. Thus, $\{P_S,P_T\}$ is not an edge of $\mathcal{D}(G)$.

On the other hand, if $i=2$, then $t$ is not a cut vertex of $H_1$, since $K_2$ is connected. Therefore, $t$ is not a cut vertex of $G_{\overline S}$. It follows that
$$
P_S + P_T= \left( \, \bigcup_{h \in S\cup \{t\}} \{x_h, y_h\}, J_{(\widetilde{H}_1)_{\overline{\{t\}}}}, J_{\widetilde{H}_2}, J_{\widetilde{G}_2}, \dots, J_{\widetilde{G}_r} \, \right).
$$
Hence, $\hgt(P_S+P_T)=\hgt(P_S)+2-1=\hgt(P_S)+1$ and $\{P_S,P_T\}$ is an edge of $\mathcal{D}(G)$.

Let now $H_1 \cap K_1 \neq \emptyset$. It follows that
\begin{equation*}
V(H_1)=V(K_2 \cup \dots \cup K_i \cup (H_1 \cap K_1) \cup \{t\}) \text{\quad and \quad} V(K_1)=V(H_2 \cup \dots \cup H_i \cup (H_1 \cap K_1) \cup \{s\})
\end{equation*}
and in this case $t$ is a cut vertex of $G_{\overline S}$. Moreover,
$$
P_S + P_T= \left( \, \bigcup_{h \in S \cup \{t\}} \{x_h, y_h\}, J_{(\widetilde{H}_1)_{\overline{\{t\}}}},J_{(\widetilde{K}_1)_{\overline{\{s\}}}}, J_{\widetilde{G}_{2}} \dots, J_{\widetilde{G}_r} \, \right).
$$
In fact, $J_{\widetilde{H}_h} \subseteq J_{(\widetilde{K}_1)_{\overline{\{s\}}}}$ and $J_{\widetilde{K}_h} \subseteq J_{(\widetilde{H}_1)_{\overline{\{t\}}}}$ for all $h=2, \dots, i$. We now compute the height of $J=J_{(\widetilde{H}_1)_{\overline{\{t\}}}} + J_{(\widetilde{K}_1)_{\overline{\{s\}}}}$. Setting $W_1=H_2 \cup \dots \cup H_i$ and $W_2=K_2 \cup \dots \cup K_i$, the ideal $J$ is the binomial edge ideal of the graph $F$ obtained from $\widetilde{W}_1 \cup \widetilde{W}_2 \cup (\widetilde{H}_1 \cap \widetilde{K}_1)$ by adding the edges $\{\{v,w\} : v \in \widetilde{H}_1 \cap \widetilde{K}_1, w \in \widetilde{W}_1 \cup \widetilde{W}_2\}$. It is easy to check that the only cut sets of $F$ are $\emptyset$ and $\widetilde{H}_1 \cap \widetilde{K}_1$. Moreover,
\begin{align*}
\hgt(P_{\widetilde{H}_1 \cap \widetilde{K}_1}(F)) &= |V(F)|+|V(\widetilde{H}_1 \cap \widetilde{K}_1)|-2 \geq |V(F)|-1 = \hgt(P_{\emptyset}(F))\\ &= |V(\widetilde{W}_1)|+|V(\widetilde{W}_2)|+|V(\widetilde{H}_1 \cap \widetilde{K}_1)|-1.
\end{align*}
Thus $\hgt(J)=|V(F)|-1=\sum_{h=1}^i |V(H_h)| -2$. Since $i \geq 2$, we get
\begin{equation*}
  \hgt(P_S+P_T) =\hgt(P_S)+2-\sum_{h=1}^i \left( |V(H_h)|-1 \right) + \sum_{h=1}^i |V(H_h)| -2=\hgt(P_S)+i>\hgt(P_S)+1.
\end{equation*}
Hence, $\{P_S,P_T\}$ is not an edge of $\mathcal{D}(G)$.
\end{proof}

\begin{remark}
Recall that a graph is \textit{$k$-connected} if it has more than $k$ vertices and the removal of any $h<k$ vertices does not disconnect the graph. In particular, every non-empty connected graph, which is not reduced to a single vertex, is $1$-connected.

Let $G$ be a connected graph such that $\mc D(G)$ is connected. If $G$ is not the complete graph, then $G$ is $1$-connected but not $2$-connected. In fact, if $G$ is $2$-connected, then $G$ does not have cut vertices and, by Theorem \ref{Dual graph} a), it follows that $P_{\emptyset}$ is an isolated vertex of the dual graph $\mathcal D(G)$, a contradiction. Notice that, if $G$ is bipartite, by Lemma \ref{P.unmixed2leaves}, it is enough to require $J_G$ to be unmixed. Nevertheless, in the non-bipartite case we need to assume $\mc D(G)$ connected. In fact, the graph $G$ in Figure \ref{F.unmixed2connected} is $2$-connected, $J_G$ is unmixed and $\mc D(G)$ consists of two isolated vertices.

We also observe that the above statement generalizes \cite[Proposition 3.10]{BN17}, since having a connected dual graph is a weaker condition (see also \cite[Corollary 2.4]{H62}). In particular, being not $2$-connected is a necessary condition for $J_G$ to be Cohen-Macaulay.

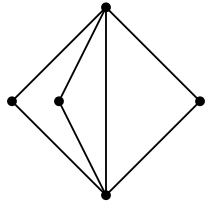
\begin{figure}[ht!]
\begin{tikzpicture}
\draw (0.,0.)-- (0.,2.5);
\draw (0.,2.5)-- (1.25,1.25);
\draw (1.25,1.25)-- (0.,0.);
\draw (0.,0.)-- (-1.25,1.25);
\draw (-1.25,1.25)-- (0.,2.5);
\draw (0.,2.5)-- (-0.625,1.25);
\draw (-0.625,1.25)-- (0.,0.);
\node (a) at (0.,0.) {};
\node (b) at (0.,2.5) {};
\node (c) at (-1.25,1.25) {};
\node (d) at (1.25,1.25) {};
\node (e) at (-0.625,1.25) {};
\end{tikzpicture}
\caption{A $2$-connected graph $G$ with $J_G$ unmixed and $\mc D(G)$ disconnected} \label{F.unmixed2connected}
\end{figure}
\end{remark}

\begin{example}\label{E.unmixedNotCM2}
For every $k \geq 4$, let $M_{k,k}$ and $M_{k-1,k}$ be the graphs defined in Example \ref{E.unmixedNotCM1}. With the same notation used there and by Theorem \ref{Dual graph}, their dual graphs are represented in Figure \ref{F.dualGraphRagni}.

\begin{figure}[ht!]
\begin{subfigure}[c]{0.3\textwidth}
\centering
\begin{tikzpicture}[scale=1.1]
\draw (0.,0.)-- (0.7071067811865476,0.7071067811865475);
\draw (0.,0.)-- (0.7071067811865476,-0.7071067811865475);
\draw (0.7071067811865476,0.7071067811865475)-- (1.4142135623730951,0.);
\draw (1.4142135623730951,0.)-- (0.7071067811865476,-0.7071067811865475);
\node[label={left:{\small $P_\emptyset$}}] (a) at (0.,0.) {};
\node[label={[label distance=-1mm]90:{\small $P_{\{2\}}$}}] (b) at (0.7071067811865476,0.7071067811865475) {};
\node[label={[label distance=-4mm]270:{\small $P_{\{2k-1\}}$}}] (c) at (0.7071067811865476,-0.7071067811865475) {};
\node[label={right:{\small $P_{\{2,2k-1\}}$}}] (d) at (1.4142135623730951,0.) {};
\node[label={[label distance=-3mm]90:{\small $P_{V_1 \setminus \{1\}}$}}] (e) at (2.5,0.7071067811865475) {};
\node[label={[label distance=-4mm]270:{\small $P_{V_2 \setminus \{2k\}}$}}] (f) at (2.5,-0.7071067811865475) {};
\end{tikzpicture}
\caption{The dual graph of $J_{M_{k,k}}$} \label{F.dualGraphM_(k,k)}
\end{subfigure}
\begin{subfigure}[c]{0.3\textwidth}
\centering
\begin{tikzpicture}[scale=1.1]
\draw (0.,0.)-- (0.7071067811865476,0.7071067811865475);
\draw (0.,0.)-- (0.7071067811865476,-0.7071067811865475);
\draw (0.7071067811865476,0.7071067811865475)-- (1.4142135623730951,0.);
\draw (1.4142135623730951,0.)-- (0.7071067811865476,-0.7071067811865475);
\draw (2.4142135623730951,0.)-- (1.4142135623730951,0.);
\node[label={left:{\small $P_\emptyset$}}] (a) at (0.,0.) {};
\node[label={[label distance=-1mm]90:{\small $P_{\{2\}}$}}] (b) at (0.7071067811865476,0.7071067811865475) {};
\node[label={[label distance=-1mm]270:{\small $P_{\{6\}}$}}] (c) at (0.7071067811865476,-0.7071067811865475) {};
\node[label={[label distance=-0.5mm]290:{\small $P_{\{2,6\}}$}}] (d) at (1.4142135623730951,0.) {};
\node[label={[label distance=-1mm]290:{\small $P_{\{2,4,6\}}$}}] (d) at (2.4142135623730951,0) {};
\node[label={[label distance=-2mm]90:{\small $P_{\{3,5\}}$}}] (e) at (2.4142135623730951,0.7071067811865475) {};
\end{tikzpicture}
\caption{The dual graph of $J_{M_{3,4}}$} \label{F.dualGraphM_(3,4)}
\end{subfigure}
\begin{subfigure}[c]{0.35\textwidth}
\centering
\begin{tikzpicture}[scale=1.1]
\draw (0.,0.)-- (0.7071067811865476,0.7071067811865475);
\draw (0.,0.)-- (0.7071067811865476,-0.7071067811865475);
\draw (0.7071067811865476,0.7071067811865475)-- (1.4142135623730951,0.);
\draw (1.4142135623730951,0.)-- (0.7071067811865476,-0.7071067811865475);
\node[label={left:{\small $P_\emptyset$}}] (a) at (0.,0.) {};
\node[label={[label distance=-1mm]90:{\small $P_{\{2\}}$}}] (b) at (0.7071067811865476,0.7071067811865475) {};
\node[label={[label distance=-4mm]270:{\small $P_{\{2k-2\}}$}}] (c) at (0.7071067811865476,-0.7071067811865475) {};
\node[label={right:{\small $P_{\{2,2k-2\}}$}}] (d) at (1.4142135623730951,0.) {};
\node[label={[label distance=-7mm]90:{\small $P_{V_1 \setminus \{1,2k-1\}}$}}] (e) at (2.5,0.7071067811865475) {};
\node[label={[label distance=-1mm]270:{\small $P_{V_2}$}}] (f) at (2.5,-0.7071067811865475) {};
\end{tikzpicture}
\caption{The dual graph of $J_{M_{k-1,k}}$, $k \!\geq\! 5$} \label{F.dualGraphM_(k-1,k)}
\end{subfigure}
\caption{} \label{F.dualGraphRagni}
\end{figure}
Thus, $J_{M_{k,k}}$ and $J_{M_{k-1,k}}$ are not Cohen-Macaulay by Hartshorne's Theorem \cite{H62}. Notice that, $M_{3,4}$ is the bipartite graph with the smallest number of vertices whose binomial edge ideal is unmixed and not Cohen-Macaulay.
\end{example}

The following technical result has several crucial consequences, see Theorem \ref{T.cutSetMinusVertex} and Theorem \ref{T.moreThan2CutVertices}. We show that, under some assumption on the graph, the intersection of two cut sets, which differ by one element and have the same cardinality, is again a cut set.

\begin{lemma}\label{INTERSECTION}
Let $G$ be a graph such that $J_G$ is unmixed. Let $S,T \in \mathcal{M}(G)$ with $|S|=|T|$ and $|S \setminus T|=1$.
\begin{itemize}
\item[$\mathrm{(i)}$] If $\{P_S,P_T\} \in \mc{D}(G)$, then $S \cap T \in \mathcal{M}(G)$.
\item[$\mathrm{(ii)}$] If $S \cup T \in \mathcal{M}(G)$ and $G$ is bipartite, then $S \cap T \in \mathcal{M}(G)$.
\end{itemize}
\end{lemma}

\begin{proof}
Let $S=(T \setminus \{t\}) \cup \{s\}$ and let $G_1, \dots, G_r$ be the connected components of $G_{\overline{S\cap T}}$. Suppose first that $s \in G_i$ and $t \in G_j$ with $i \neq j$. Let $z \in S \cap T$ such that $c_G((S \cap T) \setminus \{z\})=c_G(S \cap T)$. Since $z \in S$ and $S \in \mc M(G)$, $z$ joins at least two components of $G_{\overline S}$. Then in $G_{\overline S}$ it is only adjacent to some components of $(G_i)_{\overline{\{s\}}}$. This implies that it does not join any components in $G_{\overline T}$, a contradiction, since $T \in \mc M(G)$.

Assume now that $s,t \in G_1$ and suppose first that $r=|S \cap T|+1$. We claim that $S \cap T \in \mc M(G)$. In this case, $G_{\overline S}$ has $r+1$ connected components, say $H_1, H_2, G_2, \dots, G_r$. Consider the set
\[
Z=\{z \in S \cap T : \text{ adding } z \text{ to } G_{\overline S} \text{ it connects only } H_1 \text{ and } H_2\}.
\]
We show that $X=(S \cap T) \setminus Z \in \mc M(G)$. For every $x \in X$, we know that $c_G(S \setminus \{x\})<c_G(S)$. In particular, adding $x$ to $G_{\overline S}$, it joins some connected components and at least one of them is $G_i$ with $i \geq 2$. Hence, $c_G(X \setminus \{x\})<c_G(X)$. Moreover, $c_G(X)=|S \cap T|-|Z|+1$, by the unmixedness of $J_G$. On the other hand, by definition of $Z$ and since $S \in \mc M(G)$, it follows that $c_G(X)=r=|S|=|S \cap T|+1$. Thus, $Z=\emptyset$ and $S \cap T=X \in \mc M(G)$.

Suppose now that $H_1, \dots, H_i, G_2,\dots, G_r$ are the connected components of $G_{\overline{S}}$, with $i \geq 3$, and that $t \in H_1$. In the same way let $K_1, \dots, K_i, G_2, \dots, G_r$ be the connected components of $G_{\overline{ T}}$ and let $s \in K_1$. We show that this case cannot occur.

Following the same argument of the proof of Theorem \ref{Dual graph} c), we conclude that the connected components of $G_{\overline{S \cup T}}$ are $H_2, \dots, H_i, K_2, \dots, K_i, G_2, \dots, G_r$ and the connected components of $H_1 \cap K_1$, if it is not empty.

\medskip
{\bf (i)} \  If $H_1 \cap K_1 \neq \emptyset$, it follows that $V(H_1)=V(K_2 \cup \dots \cup K_i \cup (H_1 \cap K_1) \cup \{t\})$ and $V(K_1)=V(H_2 \cup \dots \cup H_i \cup (H_1 \cap K_1) \cup \{s\})$. In this case, $t$ is a cut vertex of $G_{\overline S}$, hence $\{P_S,P_T\}$ is not an edge of $\mc D(G)$ by Theorem \ref{Dual graph} c), a contradiction.

Let now $H_1 \cap K_1 = \emptyset$, then $V(H_1)=V(K_2 \cup \dots \cup K_i \cup \{t\})$ and $V(K_1)=V(H_2 \cup \dots \cup H_i \cup \{s\})$. Since $i \geq 3$, $t$ is a cut vertex of $H_1$, hence $\{P_S,P_T\}$ is not an edge of $\mc D(G)$ by Theorem \ref{Dual graph} c), a contradiction.

\medskip
{\bf (ii)} In this case, since both $S$ and $S \cup T$ are cut sets of $G$ and $i \geq 3$, we have that $i=3$ and $H_1 \cap K_1 = \emptyset$. Therefore, the connected components of $G_{\overline{S \cup T}}$ are $H_2, H_3, K_2, K_3, G_2, \dots, G_r$.

We know that $s$ is adjacent to $x \in H_1$ and that $s \in K_1$. Hence, $s$ is not adjacent to any vertices of $K_2$ or $K_3$. Thus, $x=t$, since $V(H_1)=V(K_2 \cup K_3 \cup \{t\})$. This means that $\{s,t\} \in E(G)$. Let
\[
Z=\{z \in S \cap T : \text{ adding } z \text{ to } G_{\overline{S \cup T}} \text{ it connects only some } H_i \text{ with some } K_j\}.
\]
Notice that, there are no vertices in $S \cap T$ that only connects $H_2$ to $H_3$ or $K_2$ to $K_3$ in $G_{\overline{S \cup T}}$. In fact, if $z \in S \cap T$ only connects $H_2$ to $H_3$ in $G_{\overline{S \cup T}}$, then $c_G(T \setminus \{z\})=c_G(T)$, a contradiction, since $T \in \mc M(G)$. The same holds for $K_2$ and $K_3$.

As above, since $S \cup T \in \mc M(G)$, it follows that $(S \cap T) \setminus Z \in \mc M(G)$ and, by the unmixedness of $J_G$, $|Z|=1$, say $Z=\{z\}$. Without loss of generality, we may assume that $z$ connects at least $H_2$ and $K_2$.

Since $s$ and $t$ are adjacent in $G$, one of them is in the same bipartition set of $z$. Without loss of generality, assume that this vertex is $t$, thus $N_{H_2}(s) \cap N_{H_2}(z)=\emptyset$.
Let
\[
A=\{x \in S \cap T : \text{ if } \{x,v\} \in E(G) \text{ for some } v \in G_1, \text{ then } v \in N_{H_2}(s)\}.
\]
Notice that, $A$ contains also all vertices of $S \cap T$ that connect only some $G_j$'s in $G_{\overline{S \cap T}}$, with $j \geq 2$. We claim that
\[
W=((S \cap T) \setminus A) \cup N_{H_2}(s) \in \mc M(G).
\]
In Figure \ref{F.W} the set $W$ is colored in gray and the circles represent the connected components of $G_{\overline{S \cup T}}$, where only some vertices are drawn.
\begin{figure}[ht!]
\vspace*{-1cm}
\begin{tikzpicture}
\tikzset{>=stealth}
\draw [fill=lightgray, rotate around={-1.4840146756621018:(5.664944996658899,2.4339553967496657)}] (5.664944996658899,2.4339553967496657) ellipse (1.0612064296793244cm and 0.2732968671960563cm);
\draw [fill=lightgray, rotate around={-11.205115725904253:(-0.8593667999432268,1.9773742698475931)}] (-0.8593667999432268,1.9773742698475931) ellipse (0.7009720565180309cm and 0.26481020565641566cm);
\draw (0.,0.)-- (1.4962128896038784,0.);
\draw(-0.7978172765551623,2.376749279956747) circle (1.069377628817238cm);
\draw (0.,0.)-- (-1.3058422083864825,2.0742360800495936);
\draw (0.,0.)-- (-0.8593281658697598,1.9674551503263051);
\draw (0.,0.)-- (-0.3725195491211828,1.8999003128338015);
\draw(2.294030166159041,2.376749279956747) circle (1.069377628817238cm);
\draw (1.4962128896038784,0.)-- (2.3555410554736387,1.967455150326305);
\draw (0.877598013336272,3.6010912225697145)-- (-1.1715637643001746,3.0340275859910766);
\draw (0.877598013336272,3.6010912225697145)-- (2.527237683383223,3.0211397760688348);
\draw (0.877598013336272,3.6010912225697145)-- (1.857071567426649,2.6602810982460654);
\draw (1.4962128896038784,0.)-- (1.8687324387250615,1.8999003128338012);
\draw(5.504321775421079,0.) circle (1.cm);
\draw(8.50718148730342,0.) circle (1.cm);
\draw (-0.3725195491211828,1.8999003128338015)-- (3.1200769398063457,0.7657730396765263);
\draw (3.1200769398063457,0.7657730396765263)-- (5.401219296043145,-0.5230079525476501);
\draw (1.4962128896038784,0.)-- (5.6412161628321655,-1.6627341461853364);
\draw (5.6412161628321655,-1.6627341461853364)-- (8.404079007925485,-0.6905494815367931);
\draw (6.50957094935594,2.4025248998012305)-- (8.043220330102715,0.4435777916204822);
\draw (8.133434999558407,2.4025248998012305)-- (5.968282932621784,0.508016841231691);
\draw (8.133434999558407,2.4025248998012305)-- (8.996918264348608,0.4306899816982404);
\draw [->,shift={(-1.2347474923589876,1.2123722920529842)}] plot[domain=1.8076664818310126:3.0243621039394015,variable=\t]({1.*0.864046465205049*cos(\t r)+0.*0.864046465205049*sin(\t r)},{0.*0.864046465205049*cos(\t r)+1.*0.864046465205049*sin(\t r)});
\draw(-0.42703443914039924,-1.4060630270366337) circle (0.8726592982516279cm);
\draw(1.9232473287442775,-1.4060630270366339) circle (0.8726592982516279cm);
\draw (0.,0.)-- (-0.6394835912170972,-1.544154975886487);
\draw (0.,0.)-- (-0.012758592590838252,-1.3741956542251292);
\draw (6.50957094935594,2.4025248998012305)-- (2.324182080252839,-1.0661443837139186);
\draw (4.8857068991534724,2.4)-- (5.139133345269087,0.49535688404980466);
\draw (4.8857068991534724,2.4)-- (2.8871723332560886,1.971878490982849);
\draw (1.4962128896038784,0.)-- (1.5089714821947164,-1.3741956542251292);
\draw (1.4962128896038784,0.)-- (2.1356964808209753,-1.5441549758864872);
\draw (0.877598013336272,3.6010912225697145)-- (-0.7988204552746206,2.757940353666628);
\draw (-0.7988204552746206,2.757940353666628)-- (-1.3058422083864825,2.0742360800495936);
\draw [->] (5.7,2.6)-- (5.7,3.1);
\node[label={300:{\footnotesize $s$}}] (a) at (0.,0.) {};
\node[label={240:{\footnotesize $t$}}] (b) at (1.5,0.) {};
\node (c) at (-1.3058422083864825,2.0742360800495936) {};
\node (d) at (-0.8593281658697598,1.9674551503263051) {};
\node (e) at (-0.3725195491211828,1.8999003128338015) {};
\node (f) at (2.3555410554736387,1.967455150326305) {};
\node (g) at (-1.1715637643001746,3.0340275859910766) {};
\node (h) at (2.527237683383223,3.0211397760688348) {};
\node (i) at (1.857071567426649,2.6602810982460654) {};
\node (j) at (1.8687324387250615,1.8999003128338012) {};
\node[label={[label distance=-1.5mm]20:{\footnotesize $x_1 \!\in\! A$}}] (k) at (3.1200769398063457,0.7657730396765263) {};
\node (l) at (5.401219296043145,-0.5230079525476501) {};
\node (m) at (8.404079007925485,-0.6905494815367931) {};
\node[label={0:{\footnotesize $w_1$}}] (n) at (4.8857068991534724,2.4) {};
\node[label={180:{\footnotesize $w_2$}}] (o) at (6.50957094935594,2.4) {};
\node[label={[label distance=-1.5mm]75:{\footnotesize $x_2 \!\in\! A$}}] (p) at (8.133434999558407,2.4) {};
\node (q) at (8.043220330102715,0.4435777916204822) {};
\node (r) at (5.968282932621784,0.508016841231691) {};
\node (s) at (8.996918264348608,0.4306899816982404) {};
\node (t) at (-0.6394835912170972,-1.544154975886487) {};
\node (u) at (-0.012758592590838252,-1.3741956542251292) {};
\node (v) at (2.324182080252839,-1.0661443837139186) {};
\node (w) at (5.139133345269087,0.49535688404980466) {};
\node (x) at (2.8871723332560886,1.971878490982849) {};
\node (y) at (1.5089714821947164,-1.3741956542251292) {};
\node (z) at (2.1356964808209753,-1.5441549758864872) {};
\node (a1) at (-0.7988204552746206,2.757940353666628) {};
\node[label={[label distance=-0.5mm]70:{\footnotesize $z \!\in\! Z$}}] (b1) at (0.877598013336272,3.6010912225697145) {};
\node[label={[label distance=-1mm]300:{\footnotesize $y \!\in\! B$}}] (c1) at (5.6412161628321655,-1.6627341461853364) {};
\draw[draw,fill = lightgray] (0.877598013336272,3.6010912225697145) circle (1mm);
\draw[draw,fill = lightgray] (5.6412161628321655,-1.6627341461853364) circle (1mm);
{\tikzstyle{every node}=[fill=white,draw=white,circle,inner sep=1pt,minimum width=0.2pt]
\node[label={\footnotesize $H_2$}] at (-1.9,2.8) {};
\node[label={\footnotesize $K_2$}] at (2.8,3.2) {};
\node[label={\footnotesize $H_3$}] at (-1.5,-2.3) {};
\node[label={\footnotesize $K_3$}] at (3,-2.3) {};
\node[label={\footnotesize $G_2$}] at (6.6,-1.2) {};
\node[label={\footnotesize $G_r$}] at (9.6,-1.2) {};
\node[label={\footnotesize $\cdots$}] at (7,-0.5) {};
\node[label={[label distance=-1.5cm]90:{\footnotesize $(S \cap T) \!\setminus\! (A \cup B \cup \{z\})$}}] at (5.5,3) {};
\node[label={[label distance=-1.4cm]270:{\footnotesize $N_{H_2}(s)$}}] at (-2.2,0.4) {};
}
\end{tikzpicture}
\caption{The set $W$ in gray} \label{F.W}
\end{figure}
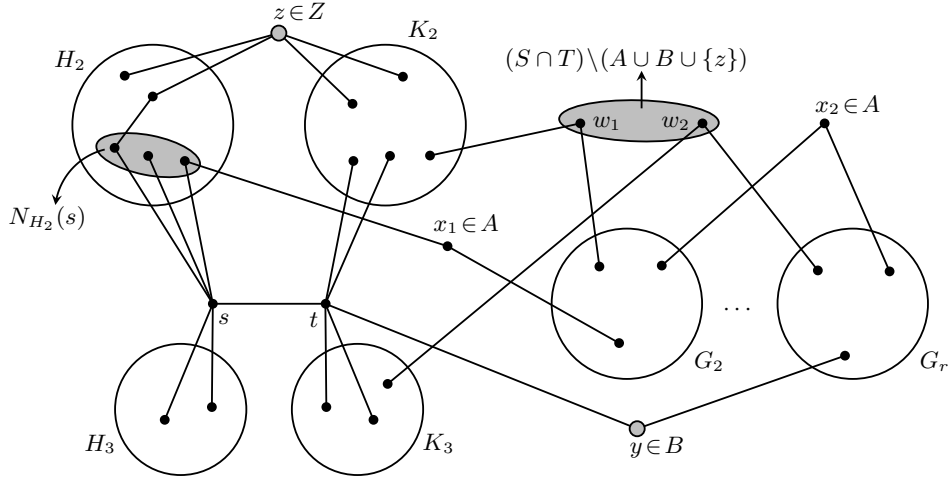

Notice that $z \in W$. Let $w \in W$. Adding $w=z$ to $G_{\overline W}$, we connect a vertex of $H_2 \setminus N_{H_2}(s)$ with $K_2$ whereas, adding $w \in N_{H_2}(s)$ to $G_{\overline W}$, we connect $s$ to $H_2 \setminus N_{H_2}(s)$. Moreover, if $w \in (S \cap T) \setminus (A \cup \{z\})$, we know that, in $G_{\overline{S \cap T}}$, $w$ connects $G_i$ for some $i \geq 2$ to a vertex $v$ of $G_1 \setminus N_{H_2}(s)$. By construction, in $G_{\overline{W}}$ the connected components containing $v$ and $G_i$ are different and $w$ still connects them. This proves that $W \in \mc M(G)$.

Since $J_G$ is unmixed, we have that $c_G(W)=|W|+1$ and a connected component of $G_{\overline W}$ is the subgraph induced on $H_3 \cup K_2 \cup K_3 \cup \{s,t\}$. Thus, removing $t$ from $G_{\overline W}$, this component splits in three components, $H_3 \cup \{s\}$, $K_2$, $K_3$. Therefore, if $W \cup \{t\}$ is a cut set of $G$, we get $c_G(W \cup \{t\}) = c_G(W) + 2= |W|+3$, which contradicts the unmixedness of $J_G$.

Hence, we may assume that $W \cup \{t\} \notin \mc M(G)$. Thus there exists $y \in N_G(t)$ that joins $t$ with only one connected component of $G_{\overline{W}}$ (i.e., $c_G((W \cup \{t\}) \setminus \{y\})=c_G(W \cup \{t\})$). In this case, we define
\[
B=\{y \in S \cap T : \{y,t\} \in E(G) \text{ and } N_G(y)\setminus \{t\} \text{ is contained in one connected component of } G_{\overline{W}}\},
\]
where $|B| \geq 1$, since $W \cup \{t\} \notin \mc M(G)$. We claim that
\[
W'=(W \setminus B) \cup \{t\} \in \mc M(G).
\]
Notice that $z \in W'$. The proof is similar to the case of $W$. We only notice that, adding $t$ to $G_{\overline{W'}}$, we connect at least $K_2$, $K_3$ and the connected component containing $s$. Moreover, each element in $B$ does not connect different connected components of $G_{\overline W}$ and any two elements of $B$ are not adjacent (since they are adjacent to $t$ and $G$ is bipartite).
Thus, $|W'|<|W \cup \{t\}|$ and
\[
c_G(W')=c_G(W \cup \{t\})=c_G(W)+2=|W|+3=|W \cup \{t\}|+2>|W'|+1,
\]
which contradicts the unmixedness of $J_G$.
\end{proof}

\begin{remark}
It could be true that, if $G$ is bipartite and $J_G$ is unmixed, then $S \cap T \in \mathcal{M}(G)$ for every $S,T \in \mathcal{M}(G)$. Both assumptions are needed: in fact, if $G$ is the graph in Figure \ref{F.CMIntersectionNotCutSet}, one can check with \texttt{Macaulay2} \cite{Mac2} that $J_G$ is Cohen-Macaulay and thus $\mathcal D(G)$ is connected. Nevertheless, $\{2,4\}, \{4,5\} \in \mc M(G)$ and $\{2,4\} \cap \{4,5\} = \{4\} \notin \mathcal M(G)$.

On the other hand, if $G$ is the cycle of length $6$ with consecutive labelled vertices, then $J_G$ is not unmixed, $\{1,3\}, \{1,5\} \in \mc M(G)$ and $\{1,3\} \cap \{1,5\} = \{1\} \notin \mathcal M(G)$.

\begin{figure}[ht!]
\begin{subfigure}[c]{0.4\textwidth}
\centering
\begin{tikzpicture}
\draw (0.,0.)-- (2.,0.);
\draw (2.,0.)-- (2.,2.);
\draw (2.,2.)-- (0.,2.);
\draw (0.,2.)-- (0.,0.);
\draw (2.,0.)-- (4.,0.);
\draw (2.,2.)-- (4.,2.);
\draw (0.,0.)-- (1.,1.);
\draw (1.,1.)-- (2.,0.);
\node[label={below:$4$}] (a) at (0.,0.) {};
\node[label={below:$5$}] (b) at (2.,0.) {};
\node[label={above:$2$}] (c) at (2.,2.) {};
\node[label={above:$3$}] (d) at (0.,2.) {};
\node[label={below:$6$}] (e) at (4.,0.) {};
\node[label={above:$1$}] (f) at (4.,2.) {};
\node[label={above:$7$}] (g) at (1.,1.) {};
\end{tikzpicture}
\caption{The graph $G$}
\end{subfigure}
\begin{subfigure}[c]{0.5\textwidth}
\centering
\begin{tikzpicture}
\draw (0.,0.)-- (1.0606601717798214,1.0606601717798212);
\draw (0.,0.)-- (1.0606601717798214,-1.0606601717798212);
\draw (1.0606601717798214,1.0606601717798212)-- (2.1213203435596424,0.);
\draw (2.1213203435596424,0.)-- (1.0606601717798214,-1.0606601717798212);
\draw (1.0606601717798214,1.0606601717798212)-- (2.560660171779821,1.0606601717798212);
\draw (2.1213203435596424,0.)-- (3.621320343559642,0.);
\draw (3.621320343559642,0.)-- (2.560660171779821,1.0606601717798212);
\draw (2.560660171779821,-1.0606601717798212)-- (1.0606601717798214,-1.0606601717798212);
\draw (2.560660171779821,-1.0606601717798212)-- (3.621320343559642,0.);
\node[label={left:$P_\emptyset$}] (a) at (0.,0.) {};
\node[label={[label distance=-1mm]90:$P_{\{2\}}$}] (b) at (1.0606601717798214,1.0606601717798212) {};
\node[label={[label distance=-1mm]270:$P_{\{5\}}$}] (c) at (1.0606601717798214,-1.0606601717798212) {};
\node[label={left:$P_{\{2,5\}}$}] (d) at (2.1213203435596424,0.) {};
\node[label={[label distance=-2mm]90:$P_{\{2,4\}}$}] (e) at (2.560660171779821,1.0606601717798212) {};
\node[label={[label distance=-2mm]270:$P_{\{4,5\}}$}] (f) at (2.560660171779821,-1.0606601717798212) {};
\node[label={right:$P_{\{2,4,5\}}$}] (g) at (3.621320343559642,0.) {};
\end{tikzpicture}
\caption{The dual graph of $G$}
\end{subfigure}
\caption{} \label{F.CMIntersectionNotCutSet}
\end{figure}
\end{remark}

The next result is important for Theorem \ref{T.CHARACTERIZATION}, since at the same time provides the equivalence b) $\Leftrightarrow$ d) and has important consequences for the proof of b) $\Rightarrow$ c).

\begin{theorem} \label{T.cutSetMinusVertex}
Let $G$ be a bipartite graph. If $\mc D(G)$ is connected, then for every non-empty $S \in \mc M(G)$, there exists $s \in S$ such that $S \setminus \{s\} \in \mc M(G)$.
\end{theorem}

\begin{proof}
By contradiction, let $T \in \mc M(G)$ such that $T \setminus \{t\} \notin \mc M(G)$ for every $t \in T$. Notice that $|T| \geq 2$, otherwise $T \setminus \{t\}=\emptyset \in \mc M(G)$.

Let $W \in \mathcal{M}(G)$, $W \neq T$, such that there exists a path $\mathcal{P}:P_T=P_{S_0},P_{S_1},\dots,P_{S_k},P_{S_{k+1}}=P_W$ in $\mc D(G)$. Assume $\mc P$ is a shortest path from $P_T$ to $P_W$.

{\bf Claim:} For $1 \leq i \leq k+1$, $|S_i|>|S_{i-1}|$. In particular, $|W|>|T|$.

We proceed by induction on $k \geq 0$.

Let $k=0$. First notice that $|W| \geq |T|$, otherwise by Theorem \ref{Dual graph} a), $W=T \setminus \{t\} \in \mc M(G)$ for some $t \in T$, a contradiction. If $|W|=|T|$, since $\{P_T,P_W\}$ is an edge of $\mc D(G)$, by Theorem \ref{Dual graph} a), we have that $W=(T \setminus \{t\}) \cup \{w\}$, for some $t \in T$ and $w \notin T$. By Lemma \ref{INTERSECTION}, we have that $W \cap T=T \setminus \{t\} \in \mc M(G)$, a contradiction. Then $|W|>|T|$.

Let $k \geq 1$. By induction, $|S_i|>|S_{i-1}|$, for every $1 \leq i \leq k$. In particular, by Theorem \ref{Dual graph} b), $S_i=T \cup \{s_1,\dots,s_i\}$ for $i=1,\dots,k$ and $s_j \notin T$, for $j=1,\dots,k$. Set $S=S_k$.

If $|W|<|S|$, then $|W|=|S|-1$ by Theorem \ref{Dual graph} a). Hence $W=S \setminus \{s\}$ for some $s \in S$.

First suppose that $s \in T$. Thus, $W=(T \setminus \{s\}) \cup \{s_1,\dots,s_k\}$. Since $|W|=|S_{k-1}|$, $|W \setminus S_{k-1}|=1$ and $W \cup S_{k-1}=S \in \mc M(G)$, by Lemma \ref{INTERSECTION} (ii) it follows that $S_{k-1} \cap W=(T \setminus \{s\}) \cup \{s_1,\dots,s_{k-1}\} \in \mc M(G)$. For every $i=1,\dots,k-2$, let $T_i=S_{i+1} \cap \cdots \cap S_{k-1} \cap W$. By induction on $i \leq k-2$, assume that $T_i \in \mc M(G)$, then $T_{i-1}=S_i \cap T_i=(T \setminus \{s\}) \cup \{s_1,\dots,s_i\} \in \mc M(G)$ by Lemma \ref{INTERSECTION} (ii), since $|S_i|=|T_i|$, $|T_i \setminus S_i|=1$ and $S_i \cup T_i=S_{i+1} \in \mc M(G)$. In particular, $T_0=S_1 \cap \cdots \cap S_{k-1} \cap W = (T \setminus \{s\}) \cup \{s_1\} \in \mc M(G)$, $|T_0|=|T|$, $|T_0 \setminus T|=1$ and $T_0 \cup T=S_1 \in \mc M(G)$. Again, by Lemma \ref{INTERSECTION} (ii), $T \cap T_0 = T \setminus \{s\} \in \mc M(G)$, a contradiction.

Now assume that $s \in S \setminus T$, where $s=s_j$ for some $j \in \{1,\dots,k\}$. Since $|W|=|S_{k-1}|$ and $|W \setminus S_{k-1}|=1$, by Lemma \ref{INTERSECTION} (ii), $S_{k-1} \cap W=S_{k-1} \setminus \{s_j\} \in \mc M(G)$. For every $i=j,\dots,k-2$, let $T_i=S_{i+1} \cap \cdots \cap S_{k-1} \cap W$. By induction on $i \leq k-2$, assume that $T_i \in \mc M(G)$, then $T_{i-1}=S_i \cap T_i=S_i \setminus \{s_j\} \in \mc M(G)$ by Lemma \ref{INTERSECTION} (ii), since $|S_i|=|T_i|$, $|T_i \setminus S_i|=1$ and $S_i \cup T_i=S_{i+1} \in \mc M(G)$. In particular, $T_{j-1}=S_j \cap \cdots \cap S_{k-1} \cap W=S_j \setminus \{s_j\}=S_{j-1} \in \mc M(G)$. Therefore,
\[
\mc P':P_{S_0}=P_T,P_{S_1},\dots,P_{S_{j-1}}=P_{T_{j-1}},P_{T_j},\dots,P_{T_{k-2}},P_W
\]
is a path from $P_T$ to $P_W$, shorter than $\mc P$, a contradiction.

If $|W|=|S|$, then $W=(S \setminus \{x\}) \cup \{y\}$ for some $x,y$. If $x \in T$, then $W=(T \setminus \{x\}) \cup \{s_1,\dots,s_k,y\}$. By Lemma \ref{INTERSECTION} (i), $W \cap S=(T \setminus \{x\}) \cup \{s_1,\dots,s_k\} \in \mc M(G)$. We may proceed in a similar way to the case $|W|<|S|$, setting $T_i=S_{i+1} \cap \cdots \cap S_k \cap W$ for $i=1,\dots,k-1$.

Now assume $x \in S \setminus T$, where $x=s_j$ for some $j \in \{1,\dots,k\}$. Since $|W|=|S|$, by Lemma \ref{INTERSECTION} (i), $S \cap W=S \setminus \{x\} \in \mc M(G)$. Again, we may proceed as in the case $|W|<|S|$, setting $T_i=S_{i+1} \cap \cdots \cap S_k \cap W$ for $i=j,\dots,k-1$.

In both cases we find a contradiction. In conclusion, we proved that, if there exists a path from $P_T$ to $P_W$ in $\mc D(G)$, then $|W|>|T| \geq 2$. Thus, there is no path from $P_T$ to $P_{\emptyset}$ in $\mc D(G)$, hence $\mc D(G)$ is disconnected.
\end{proof}

Using the following result, we may reduce to consider bipartite graphs $G$ with exactly two cut vertices and $\mathcal{D}(G)$ connected.

\begin{theorem}\label{T.moreThan2CutVertices}
Let $G$ be a bipartite graph with at least three cut vertices and such that $J_G$ is unmixed.
\begin{itemize}
\item[\rm a)] There exist $G_1$ and $G_2$ such that $G=G_1 \ast G_2$ or $G=G_1 \circ G_2$.
\item[\rm b)] If $\mathcal{D}(G)$ is connected, then $\mc D(G_1)$ and $\mc D(G_2)$ are connected.
\end{itemize}
\end{theorem}

\begin{proof}
a) By Proposition \ref{P.unmixed2leaves}, $G$ has exactly two leaves. Let $v$ be a cut vertex that is not a neighbour of a leaf and let $H_1$ and $H_2$ be the connected components of $G_{\overline{\{v\}}}$.
If $v$ is a leaf of both $G_{V(H_1) \cup \{v\}}$ and $G_{V(H_2) \cup \{v\}}$, then $G=G_{V(H_1) \cup \{v\}} \ast G_{V(H_2) \cup \{v\}}$.

Assume that $v$ is not a leaf of $G_{V(H_1) \cup \{v\}}$ and of $G_{V(H_2) \cup \{v\}}$. Then, given two new vertices $w_1$ and $w_2$, for $i=1,2$ we set $G_i$ to be the graph $(G_{V(H_i) \cup \{v\}}) \cup \{v,w_i\}$. It follows that $G=G_1 \circ G_2$.

Now assume by contradiction that $v$ is a leaf of $G_{V(H_2) \cup \{v\}}$, but not of $G_{V(H_1)  \cup \{v\}}$, and let $w$ be the only neighbour of $v$ in $G_{V(H_2)  \cup \{v\}}$. Hence, $w$ is a cut vertex of $G$ and we may assume that it is not a leaf of $G_{V(H_2)}$, otherwise $G= G_{V(H_1) \cup \{v,w\}} \ast G_{V(H_2)}$.

The graphs $G_{V(H_1)\cup \{v\}}$ and $G_{V(H_2)}$ are bipartite with bipartitions $V_1 \sqcup V_2$ and $W_1 \sqcup W_2$, respectively. Without loss of generality, assume that $v \in V_1$ and $w \in W_1$ and let $S=V_1 \setminus \{\ell : \ell \text{ is a leaf of } G\}$. This is a cut set of $G$: indeed in $G_{\overline S}$ all vertices of $V_2$ are either isolated or connected with only one leaf of $G$, hence every element of $S$ connects at least one vertex of $V_2$ with some other connected component. Therefore, since $J_G$ is unmixed, $G_{\overline S}$ has $|S|+1$ connected components, $G_{V(H_2)}$ is one of them and the vertices of $G_{V(H_1)}$ not in $S$ form the remaining $|S|$ connected components. In the same way, the set $T=W_1 \setminus \{\ell : \ell \text{ is a leaf of } G\} \in \mathcal{M}(G)$ and $G_{\overline T}$ consists of the connected component $G_{V(H_1)\cup \{v\}}$ and of $|T|$ connected components on the vertices of $G_{V(H_2)}$ that are not in $T$. Notice that $S \cup T$ is a cut set of $G$: in fact, adding either $v$ or $w$ to $G_{\overline{S \cup T}}$, we join at least two connected components, since $v$ is not a leaf of $G_{V(H_1)\cup \{v\}}$ and $w$ is not a leaf of $G_{V(H_2)}$. Then $G_{\overline{S \cup T}}$ has $|S|$ connected components on the vertices of $G_{V(H_1)\cup \{v\}}$ and $|T|$ on the vertices of $G_{V(H_2)}$. Hence, $c_G(S \cup T)=|S|+|T|$, a contradiction.

\medskip
b) We prove the statement for $G_1$, the argument for $G_2$ is the same. Let $P_S$ be the primary components of $J_{G_1}$, $S_0 \in \mc M(G_1)$ and $k=|S_0|$. Thus, $S_0 \in \mc M(G)$ by Theorems \ref{T.type2aCM} and \ref{T.type2bBipartiteUnmixed}. Moreover, by Theorem \ref{T.cutSetMinusVertex}, there exists $s_1 \in S_0$ such that $S_1=S_0 \setminus \{s_1\} \in \mc M(G)$. Applying repeatedly Theorem \ref{T.cutSetMinusVertex}, we find a finite sequence of cut sets $S_2=S_0 \setminus \{s_1,s_2\}, S_3=S_0 \setminus \{s_1,s_2,s_3\}, \dots, S_k=S_0 \setminus S_0=\emptyset \in \mc M(G)$. Notice that $S_i \in \mc M(G_1)$ for $i=1,\dots,k$ and, by Theorem \ref{Dual graph}, $\{P_{S_i},P_{S_{i+1}}\}$ is an edge of $\mc M(G_1)$ for $i=1,\dots,k-1$. Hence,
\[
\mc P: P_{S_0},P_{S_1},P_{S_2},\dots,P_k=P_\emptyset,
\]
is a path from $P_S$ to $P_\emptyset$ in $\mc D(G_1)$. Therefore, $\mc D(G_1)$ is connected.
\end{proof}

\begin{remark}
If the graph $G$ is not bipartite, Theorem \ref{T.moreThan2CutVertices} a) does not hold. For instance, the ideal $J_G$ of the graph in Figure \ref{due triangoli} is unmixed, indeed Cohen-Macaulay, and $G$ has four cut vertices, but it is not possible to split it using the operations $\ast$ and $\circ$.

\begin{figure}[ht!]
\begin{tikzpicture}[scale=0.7]
\node (a) at (0,0) {};
\node (b) at (2,0) {};
\node (c) at (3,1.5) {};
\node (d) at (4,0) {};
\node (e) at (6,0) {};
\node (f) at (7,1.5) {};
\node (g) at (8,0) {};
\node (h) at (10,0) {};
\draw (0,0) -- (2,0) -- (4,0) -- (6,0) -- (8,0) -- (10,0)
(2,0) -- (3,1.5) -- (4,0)
(6,0) -- (7,1.5) -- (8,0);
\end{tikzpicture}
\caption{} \label{due triangoli}
\end{figure}
\end{remark}

The remaining part of the section is useful to prove that a bipartite graph $G$ with exactly two cut vertices and $\mc D(G)$ connected is of the form $F_m$.

\begin{corollary} \label{C.cutSetsContainCutVertex}
Let $G$ be a bipartite graph such that $\mc D(G)$ is connected. Then every non-empty cut set $S \in \mc M(G)$ contains a cut vertex.
\end{corollary}

\begin{proof}
Let $S \in \mc M(G)$ and $k=|S|$. We may assume $k \geq 2$. By Theorem \ref{T.cutSetMinusVertex}, there exists $s \in S$ such that $T=S \setminus \{s\} \in \mc M(G)$. By induction, $T$ contains a cut vertex and the claim follows.
\end{proof}

\begin{remark}
All assumptions in Theorem \ref{T.moreThan2CutVertices} and Corollary \ref{C.cutSetsContainCutVertex} are needed. In fact, both claims do not hold if we only assume $G$ bipartite but $\mathcal D(G)$ is not connected. For instance, let $G=M_{3,4}$. Then $\{3,5\}$ is a cut set that does not contain any cut vertex (see Example \ref{E.unmixedNotCM1}).

On the other hand, both results do not hold if $\mc D(G)$ is connected but $G$ is not bipartite. For example, if $G$ is the graph in Figure \ref{F.unmixedDualConnected}, then $\{3,4\} \in \mathcal M(G)$, but $3$ and $4$ are not cut vertices of $G$.
\end{remark}

\begin{corollary} \label{C.adjacentCutVertices}
Let $G$ be a bipartite graph with bipartition $V_1 \sqcup V_2$ and with exactly two cut vertices $v_1$ and $v_2$. If $\mathcal D(G)$ is connected, then $\{v_1,v_2\} \in E(G)$. In particular $|V_1|=|V_2|$.
\end{corollary}

\begin{proof}
Let $f_i$ be the leaf adjacent to $v_i$ for $i=1,2$. Assume that $\{v_1,v_2\} \notin E(G)$. Then $S_i=N_G(v_i) \setminus \{f_i\}$ is a cut set of $G$ for $i=1,2$. Moreover, $S_1$ and $S_2$ do not contain cut vertices. By Corollary \ref{C.cutSetsContainCutVertex} it follows that $\mathcal D(G)$ is disconnected, a contradiction. The last part of the claim follows from Remark \ref{unmixed cut vertices}.
\end{proof}

\begin{lemma} \label{L.vertexDegree2}
Let $G$ be a bipartite graph with bipartition $V_1 \sqcup V_2$, $|V_1|=|V_2|$ and with exactly two cut vertices. If $\mathcal{D}(G)$ is connected, then there exists a vertex of $G$ with degree $2$.
\end{lemma}

\begin{proof}
Suppose by contradiction that all the vertices of $G$, except the two leaves, have degree greater than $2$. Let $f$ be the only leaf of $G$ in $V_1$ and consider $T=V_1 \setminus \{f\}$. Clearly $G_{\overline T}$ is the disjoint union of $|V_2|-1$ isolated vertices and the edge $\{v_2,f\}$, where $v_2 \in V_2$ is a cut vertex. Therefore, $T$ is a cut set and we claim that it is an isolated vertex in $\mc D(G)$.

Notice that $T$ is not contained in any other cut set. Moreover, suppose that $S$ is a cut set of $G$ such that $S \subset T$ and $T \setminus S= \{v\}$. Since $S \subset V_1$, it follows that $\deg_{G_{\overline S}}(v)>2$. Then $c_G(S) = c_G(T \setminus \{ v \}) \leq c_G(T) -2=|V_1|-2$, since $G_{\overline T}$ consists of isolated vertices and one edge. This contradicts the unmixedness of $J_G$.

Finally, let $T'$ be a cut set such that $T \setminus T'=\{v\}$ and $T' \setminus T=\{v'\}$. If we set $S= T \setminus \{v\}=T' \setminus \{v'\}$, it follows that $v'$ has to be a cut vertex of $G_{\overline S}$. As consequence, $v'=v_2$ is the cut vertex in $V_2$, and $\{v,v'\} \in E(G)$. On the other hand, as before, $G_{\overline S}$ has at most $|V_2|-2$ connected components, then $c_G(T')=c_{G}(S)+1 \leq |V_2|-1$. This contradicts the unmixedness of $J_G$, because $|T'|=|V_2|-1$. Therefore, Theorem \ref{Dual graph} implies that $T$ is an isolated vertex in $\mathcal{D}(G)$ against our assumption.
\end{proof}

\begin{proposition}\label{P.preconoconnesso}
Let $H$ be a bipartite graph with bipartition $V_1 \sqcup V_2$ and $|V_1|=|V_2|$. Let $v$ and $f$ be two new vertices and let $G$ be the bipartite graph with $V(G)=V(H) \cup \{v,f\}$ and $E(G)=E(H) \cup \{\{v,x\} : x \in V_1 \cup \{f\} \}$. If $\mathcal{D}(G)$ is connected, then $\mathcal{D}(H)$ is connected.
\end{proposition}

\begin{proof}
Let $f_2$ be the leaf of $G$ in $V_2$ and $w$ its only neighbour, which is a cut vertex. Lemmas \ref{maximal degree} b) and \ref{L.vertexDegree2} imply that there is a vertex with degree $2$ in $G$. Thus, by Proposition \ref{Semiconi},
\[
\mathcal M(G) = \{\emptyset, V_1\} \cup \{S \cup \{v\} : S \in \mathcal M(H) \} \cup \{T \subset V_1 : T \in \mathcal M(H)\}.
\]
Let us denote by $P_S$ the primary components of $J_G$ and by $Q_S$ those of $J_H$. Using Theorem \ref{Dual graph}, we can give a complete description of the edges of $\mc D(G)$:
\begin{itemize}
\item[(i)] $\{P_\emptyset, P_T\} \in E(\mathcal{D}(G))$ if and only if either $T=\{v\}$ or $T=\{w\}$,
\item[(ii)] $\{P_{V_1},P_T\} \in E(\mathcal{D}(G))$ if and only if either $T=V_1 \setminus \{f_1\}$ or $T=(V_1 \setminus \{f_1\}) \cup \{v\}$;
\item[(iii)] if $S_1, S_2 \in \mathcal{M}(H)$, then $\{P_{S_1 \cup\{v\}}, P_{S_2 \cup \{v\}} \} \in E(\mathcal{D}(G))$ if and only if $\{Q_{S_1},Q_{S_2}\} \in E(\mathcal{D}(H))$;
\item[(iv)] if $T_1, T_2 \in \mathcal{M}(G)$ are strictly contained in $V_1$, then we have $\{P_{T_1},P_{T_2}\} \in E(\mathcal{D}(G))$ if and only if $\{Q_{T_1},Q_{T_2}\} \in E(\mathcal{D}(H))$;
\item[(v)] if $S,T \in \mathcal{M}(H)$ and $T \subsetneq V_1$, then $\{P_{S \cup \{v\}}, P_T\} \in E(\mathcal{D}(G))$ if and only if $S=T$.
\end{itemize}

If $S \in \mc M(H)$, it is enough to prove that $Q_S$ is in the same connected component as $Q_\emptyset$ in $\mathcal{D}(H)$. By (iii), this is equivalent to prove that in $\mathcal{D}(G)$ there exists a path $P_{\{v\}} = P_{U_1}, P_{U_2}, \dots, P_{U_r} = P_{S \cup \{v\}}$ such that $U_i$ contains $v$ for all $i$. Since $\mathcal{D}(G)$ is connected, we know that there exists a path $\mc P$ from $P_{\{v\}}$ to $P_{S \cup \{v\}}$. We first note that, if $\mc P$ contains $P_\emptyset$ or $P_{V_1}$, we may avoid them: in fact, by (i) and (ii), they only have two neighbours; for $P_{V_1}$ they are adjacent by (v), whereas we may replace $P_{\emptyset}$ with $P_{\{v,w\}}$ by (iii) and (v). Let $i$ be the smallest index for which $U_i$ does not contain $v$. This means that $U_i \subsetneq V_1$ and $U_{i-1}=U_i \cup \{v\}$ by (v). Moreover, $U_{i+1}$ does not contain $v$, otherwise it would be equal to $U_{i-1}$ (again by (v)). Therefore, $U_{i+1} \subsetneq V_1$ and $\{Q_{U_i}, Q_{U_{i+1}}\} \in E(\mathcal{D}(H))$ by (iv). Thus,  replacing $U_i$ with $U_{i+1} \cup \{v\}$ in $\mc P$, we get a new path from $P_{\{v\}}$ to $P_{S \cup \{v\}}$, by (iii) and (iv). Repeating the same argument finitely many times, we eventually find a path from $P_{\{v\}}$ to $P_{S \cup \{v\}}$ that involves only cut sets containing $v$. Thus $\mathcal{D}(H)$ is connected by (iii).
\end{proof}

\vspace{3mm}

\section{The main theorem} \label{S.mainTheorem}

In this section we prove the main theorem of the paper and give some applications.
\clearpage

\begin{theorem} \label{T.CHARACTERIZATION}
Let $G$ be a connected bipartite graph. The following properties are equivalent:
\begin{itemize}
  \item[{\rm a)}] $J_G$ is Cohen-Macaulay;
  \item[{\rm b)}] the dual graph $\mathcal D(G)$ is connected;
  \item[{\rm c)}] $G=A_1 \ast A_2 \ast \cdots \ast A_k$, where $A_i=F_m$ or $A_i=F_{m_1} \circ \cdots \circ F_{m_r}$, for some $m \geq 1$ and $m_j \geq 3$;
  \item[{\rm d)}] $J_G$ is unmixed and for every non-empty $S \in \mc M(G)$, there exists $s \in S$ such that $S \setminus \{s\} \in \mc M(G)$.
\end{itemize}
\end{theorem}

\begin{proof}
The implication a) $\Rightarrow$ b) follows by Hartshorne's Connectedness Theorem \cite[Proposition 1.1, Corollary 2.4, Remark 2.4.1]{H62}.

b) $\Rightarrow$ c): We may assume that $G$ has more than two vertices. Recall that, since $\mathcal{D}(G)$ is connected, then $J_G$ is unmixed. By Proposition \ref{P.unmixed2leaves}, $G$ has exactly two leaves, hence at least two cut vertices $v_1,v_2$, which are their neighbours. We proceed by induction on the number $h \geq 2$ of cut vertices of $G$.

Let $h=2$. We claim that $G=F_m$, for some $m \geq 2$. Let $V(G)=V_1 \sqcup V_2$ be the bipartition of the vertex set of $G$. By Corollary \ref{C.adjacentCutVertices}, we have that $\{v_1,v_2\} \in E(G)$ and $|V_1|=|V_2|$, with $v_i \in V_i$ for $i=1,2$. We proceed by induction on $m=|V_1|=|V_2|$. If $m=2$, then $G=F_2$. Let $m>2$ and consider the graph $H$ obtained removing $v_2$ and the leaf adjacent to it. Lemma \ref{maximal degree} b) implies that $v$ has degree $m$ and $H$ has exactly two cut vertices, whereas by Proposition \ref{P.preconoconnesso}, $\mc{D}(H)$ is connected. Hence, by induction, it follows that $H=F_{m-1}$ and $G=F_m$ by construction.

Assume now $h>2$. Let $v$ be a cut vertex of $G$ such that $v \neq v_1,v_2$. By Theorem \ref{T.moreThan2CutVertices}, there exist two graphs $G_1$ and $G_2$ such that $G=G_1 \ast G_2$ or $G=G_1 \circ G_2$ and $\mc{D}(G_1), \mc{D}(G_2)$ are connected. If $G=G_1 \ast G_2$, by induction they are of the form $A_1 \ast A_2 \ast \cdots \ast A_k$, for some $k \geq 1$, where $A_i=F_m$, with $m \geq 1$, or $A_i=F_{m_1} \circ \cdots \circ F_{m_r}$, with $m_j \geq 3$ for $j=1,\dots,r$.

On the other hand, if $G=G_1 \circ G_2$, it follows that $G_1=A_1 \ast A_2 \ast \cdots \ast A_s$ and $G_2=B_1 \ast B_2 \ast \cdots \ast B_t$, where each $A_i$ and $B_i$ are equal to $F_m$, for some $m \geq 1$, or to $F_{m_1} \circ \cdots \circ F_{m_r}$, with $m_j \geq 3$ for $j=1,\dots,r$. By Theorem \ref{T.type2bBipartiteUnmixed}, it follows that if $A_s=F_m$ or $B_1=F_m$, then $m \geq 3$.

c) $\Rightarrow$ a): Let $G$ be a graph as in c). We proceed by induction on $k \geq 1$.

If $k=1$, then $G=F_m$ for some $m \geq 1$, or $G=F_{m_1} \circ \cdots \circ F_{m_r}$, with $m_j \geq 3$ for $j=1,\dots,r$. In the first case the claim follows from Proposition \ref{P.type1CM}, in the latter from Theorem \ref{T.type2bCMbipartite}.

Let $k>1$ and consider the graphs $G_1=A_1 \ast A_2 \ast \cdots \ast A_{k-1}$ and $G_2=A_k$. By induction, $J_{G_1}$ is Cohen-Macaulay and, by the previous argument, also $J_{G_2}$ is Cohen-Macaulay. Then, the claim follows from Theorem \ref{T.type2aCM}.

b) $\Leftrightarrow$ d): The first implication follows from Theorem \ref{T.cutSetMinusVertex}. Conversely, let $S \in \mc{M}(G)$, $S \neq \emptyset$, and $P_S$ be the primary components of $J_G$. It suffices to show that there exists a path from $P_{\emptyset}$ to $P_S$. If $|S|=1$, the claim follows by Theorem \ref{Dual graph} b). If $|S|>1$, by assumption, there exists $s \in S$ such that $S \setminus \{s\} \in \mc{M}(G)$ and, by induction, there exists a path from $P_{\emptyset}$ to $P_{S \setminus \{s\}}$. Thus, Theorem \ref{Dual graph} b) implies that $\{P_{S \setminus \{s\}},P_S\}$ is an edge of $\mc{D}(G)$.
\end{proof}

Theorem \ref{T.CHARACTERIZATION} can be restated in the following way. Let $G$ be a connected bipartite graph. If it has exactly two cut vertices, then $J_G$ is Cohen-Macaulay if and only if $G=F_m$ for some $m \geq 1$. If it has more than two cut vertices, then $J_G$ is Cohen-Macaulay if and only if there exist two bipartite graphs $G_1,G_2$ such that $J_{G_1}, J_{G_2}$ are Cohen-Macaulay and $G=G_1 \ast G_2$ or $G=G_1 \circ G_2$.

\smallskip
Figure \ref{F.classification} shows a graph $G$ obtained by a sequence of operations $\ast$ and $\circ$ on a finite set of graphs of the form $F_m$. More precisely, $G=F_3 \ast F_3 \circ F_4 \ast F_1 \ast F_3 \circ F_3$ and $v_i$ denotes the only common vertex between two consecutive blocks. By Theorem \ref{T.CHARACTERIZATION}, $J_G$ is Cohen-Macaulay.

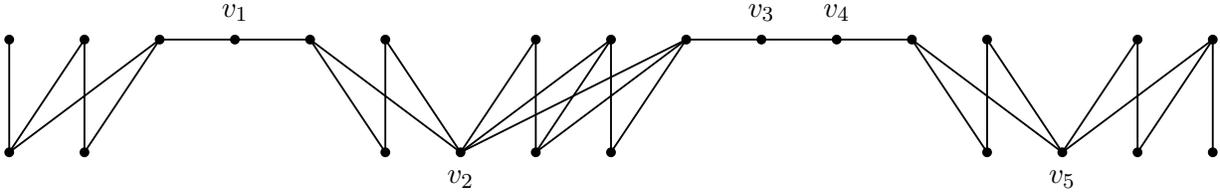
\begin{figure}[ht!]
\begin{tikzpicture}
\draw (7.,0.)-- (8.,1.5);
\draw (7.,0.)-- (9.,1.5);
\draw (8.,0.)-- (9.,1.5);
\draw (7.,0.)-- (10.,1.5);
\draw (8.,0.)-- (10.,1.5);
\draw (9.,0.)-- (10.,1.5);
\draw (6.,1.5)-- (7.,0.);
\draw (5.,1.5)-- (7.,0.);
\draw (6.,1.5)-- (6.,0.);
\draw (6.,0.)-- (5.,1.5);
\draw (3.,1.5)-- (2.,0.);
\draw (1.,0.)-- (3.,1.5);
\draw (2.,1.5)-- (1.,0.);
\draw (2.,1.5)-- (2.,0.);
\draw (1.,1.5)-- (1.,0.);
\draw (13.,1.5)-- (14.,0.);
\draw (13.,1.5)-- (15.,0.);
\draw (14.,1.5)-- (15.,0.);
\draw (16.,1.5)-- (15.,0.);
\draw (15.,0.)-- (17.,1.5);
\draw (17.,1.5)-- (16.,0.);
\draw (17.,1.5)-- (17.,0.);
\draw (8.,1.5)-- (8.,0.);
\draw (9.,1.5)-- (9.,0.);
\draw (14.,1.5)-- (14.,0.);
\draw (16.,1.5)-- (16.,0.);
\draw (13.,1.5)-- (10.,1.5);
\draw (3.,1.5)-- (5.,1.5);
\node (a) at (6.,0.) {};
\node[label={below:$v_2$}] (b) at (7.,0.) {};
\node (c) at (8.,0.) {};
\node (d) at (9.,0.) {};
\node (e) at (6.,1.5) {};
\node (f) at (8.,1.5) {};
\node (g) at (9.,1.5) {};
\node (h) at (10.,1.5) {};
\node (i) at (5.,1.5) {};
\node[label={above:$v_1$}] (j) at (4.,1.5) {};
\node (k) at (2.,0.) {};
\node (l) at (3.,1.5) {};
\node (m) at (1.,0.) {};
\node (n) at (1.,1.5) {};
\node (o) at (2.,1.5) {};
\node[label={above:$v_3$}] (p) at (11.,1.5) {};
\node[label={above:$v_4$}] (q) at (12.,1.5) {};
\node (r) at (13.,1.5) {};
\node (s) at (14.,0.) {};
\node[label={below:$v_5$}] (t) at (15.,0.) {};
\node (u) at (16.,0.) {};
\node (v) at (17.,0.) {};
\node (w) at (14.,1.5) {};
\node (x) at (16.,1.5) {};
\node (y) at (17.,1.5) {};
\end{tikzpicture}
\caption{The graph $G=F_3 \ast F_3 \circ F_4 \ast F_1 \ast F_3 \circ F_3$} \label{F.classification}
\end{figure}

\bigskip
It is interesting to notice that, Theorem \ref{T.CHARACTERIZATION} gives, at the same time, a classification of other known classes of Cohen-Macaulay binomial ideals associated with graphs. We recall that, given a graph $G$, the \textit{Lov\'asz-Saks-Schrijver ideal} $L_G$ (see \cite{HMMW15}), the \textit{permanental edge ideal} $\Pi_G$ (see \cite[Section 3]{HMMW15}) and the \textit{parity binomial edge ideal} $\mathcal I_G$ (see \cite{KSW16}) are defined respectively as
\begin{gather*}
L_G=(x_ix_j + y_iy_j : \{i,j\} \in E(G)), \\
\Pi_G=(x_iy_j + x_jy_i : \{i,j\} \in E(G)), \\
\mc I_G=(x_ix_j - y_iy_j : \{i,j\} \in E(G)).
\end{gather*}

\begin{corollary} \label{C.otherBinomialIdeals}
Let $G$ be a bipartite connected graph. Then Theorem {\rm \ref{T.CHARACTERIZATION}} holds for $L_G$, $\Pi_G$ and $\mathcal I_G$.
\end{corollary}

\begin{proof}
Let $G$ be a bipartite graph  with bipartition $V(G)=V_1 \sqcup V_2$. Then the binomial edge ideal $J_G$ can be identified respectively with $L_G$, $\Pi_G$ and $\mathcal I_G$ by means of the isomorphisms induced by:
\[
\begin{array}{ccc}
(x_i,y_i) \xmapsto{L_G} \begin{cases}
(x_i,y_i) & \text{if } i \in V_1\\
(y_i,-x_i) & \text{if } i \in V_2
\end{cases},
& \hspace{1mm}
(x_i,y_i) \xmapsto{\Pi_G} \begin{cases}
(x_i,y_i) & \text{if } i \in V_1\\
(-x_i,y_i) & \text{if } i \in V_2
\end{cases},
& \hspace{1mm}
(x_i,y_i) \xmapsto{\mathcal I_G} \begin{cases}
(x_i,y_i) & \text{if } i \in V_1\\
(y_i,x_i) & \text{if } i \in V_2
\end{cases}.
\end{array}
\]
Notice that the first transformation is more general than the one described in \cite[Remark 1.5]{HMMW15}.

Thus, for bipartite graphs, these four classes of binomial ideals are essentially the same and Theorem \ref{T.CHARACTERIZATION} classifies which of these ideals are Cohen-Macaulay.
\end{proof}

\medskip
As a final application, we show that \cite[Conjecture 1.6]{BV15} holds for Cohen-Macaulay binomial edge ideals of bipartite graphs. Recall that the \textit{diameter}, $\mathrm{diam}(G)$, of a graph $G$ is the maximal distance between two of its vertices. A homogeneous ideal $I$ in $A=K[x_1,\cdots,x_n]$ is called \textit{Hirsch} if $\mathrm{diam}(\mc{D}(I)) \leq \hgt(I)$. In \cite{BV15}, the authors conjecture that every Cohen-Macaulay homogeneous ideal generated in degree two is Hirsch.

\begin{corollary} \label{C.Hirsch}
Let $G$ be a bipartite connected graph such that $J_G$ is Cohen-Macaulay. Then $J_G$ is Hirsch.
\end{corollary}

\begin{proof}
Let $S \in \mc{M}(G)$ be a cut set of $G$ and let $n=|V(G)|$. We may assume $n \geq 3$, otherwise $\mc{D}(J_G)$ is a single vertex. Since $J_G$ is unmixed, $G_{\overline S}$ has exactly $|S|+1$ connected components and we claim that $|S| \leq \lceil \frac{n}{2} \rceil -1$.  In fact, if $|S| \geq \lceil \frac{n}{2} \rceil$, we would have
$$|V(G)| \geq |S|+|S|+1 \geq \left\lceil \frac{n}{2} \right\rceil + \left\lceil \frac{n}{2} \right\rceil +1 \geq \frac{n}{2} + \frac{n}{2} +1= n+1,$$
a contradiction. Consider now another cut set $T$  of $G$. By Theorem \ref{T.CHARACTERIZATION} d), it follows that there is a path connecting $P_S$ and $P_T$, containing $P_{\emptyset}$ and with length $|S|+|T| \leq 2(\lceil \frac{n}{2} \rceil -1) \leq n-1$. Thus, $\mathrm{diam}(\mc D(J_G)) \leq n-1=\hgt(J_G)$. \qedhere
\end{proof}

\section*{Acknowledgments}
The authors acknowledge the extensive use of the software \texttt{Macaulay2} \cite{Mac2} and \texttt{Nauty} \cite{Nauty}. They also thank Giancarlo Rinaldo for pointing out an error in Remark 3.1 in a previous version of the manuscript.

\end{document}